\documentclass[10pt, reqno]{article}
\usepackage{amsmath, amsthm, amscd, amsfonts, amssymb, graphicx, color}
\usepackage[T1]{fontenc}
\usepackage[utf8]{inputenc}
\usepackage{float}
\usepackage{graphicx}
\usepackage{subcaption}
\usepackage[bookmarksnumbered, colorlinks, plainpages]{hyperref}

\newtheorem{theorem}{Theorem}[section]
\newtheorem{lemma}[theorem]{Lemma}
\newtheorem{proposition}[theorem]{Proposition}
\newtheorem{corollary}[theorem]{Corollary}
\theoremstyle{definition}
\newtheorem{definition}[theorem]{Definition}
\newtheorem{example}[theorem]{Example}

\theoremstyle{remark}
\newtheorem{remark}[theorem]{Remark}
\numberwithin{equation}{section}

\begin{document}
	\begin{center}
		{\large{\textbf{ Curvatures in contravariant warped product space}}}
	\end{center}
	\begin{center}
		{ Pankaj Kumar, Buddhadev Pal and Santosh Kumar\footnote{The Second author is supported by UGC JRF of India, Ref. No. 1068/ ( CSIR-UGC NET JUNE 2019) \\$^{1}$Corresponding author}}
	\end{center}
	\begin{center}
		Department of Mathematics\\
		Institute of Science\\
		Banaras Hindu University\\
		Varanasi-221005, India\\
			E-mail: pankaj.kumar14@bhu.ac.in\\
		E-mail: pal.buddha@gmail.com\\
		E-mail: thakursantoshbhu@gmail.com
	\end{center}
	\begin{center}
		\textbf{Abstract}
		\end{center}
		In this article,  we introduce  the sectional curvature in contravariant warped product space $(M= M_{1}\times_{f_{1}}M_{2},\Pi,g^{f_{1}})$, where $\Pi=\Pi_1+\nu_{1}\Pi_2$). After that we find the sectional curvature of $M$ for which  $M_{1}$ and $M_{2}$ are Poisson manifolds  of positive sectional curvatures. In dual space of $M$, we introduce the notion of  null, spacelike, timelike $1 -$ forms and  then by using these forms, qualar curvature is defined. Finally, as an  examples we obtain  the sectional curvatures for $M_{1} = H_{1}^2$, $M_{2} = S_{0}^2 , E_{2}^2$   and qualar curvature for $M$. 
	
	
	\section{\textbf{Introduction}}
	
	 R.L. Bishop and B. O’Neill \cite{lota}, provided the notion of warped product space to establish examples on complete Riemannian manifolds with negative sectional curvature.
	Suppose $(M_{1},\bar{g_{1}})$ and $(M_{2},\bar{g_{2}})$ are two pseudo-Riemannian manifolds with a positive smooth function $f_{1}$ on $M_{1}$. If $\pi_{1} : M_{1} \times M_{2} \rightarrow M_{1}$ and $\pi_{2} : M_{1} \times M_{2} \rightarrow M_{2}$ are the natural projections then the warped product $M=M_{1} \times_{f_1} M_{2}$ is the product manifold  $M_{1}\times M_{2}$ equipped with the metric
	\begin{align*}
	\bar{g}^{f_{1}}=\pi_{1}^*(\tilde{g_{1}})+ (f_{1}\circ\pi_1)^2\pi_{2}^*(\tilde{g_{2}}),
	\end{align*}
		where $^*$ stands for the pull-back operator. The ordered pair $(M,g)$ is said to be warped product space. Here $M_{1}$, $M_{2}$ and $f_{1}$ are said to be base space, fiber space and warping function of $M$ respectively.
\par
Poisson \cite{sdp}, introduced a bracket in classical dynamics known as Poisson bracket. After that Lie \cite{sli}, started to study the geometric properties of this bracket. From then on Poisson geometry has become an active field of research. I. Vaisman \cite{ivs}, provided the concept of contravariant derivative on a Poisson manifold. Afterward R. L. Fernandes \cite{rfl}, characterized many results on Poisson manifold with contravariant connection. M. Boucetta \cite{mbo,mbo2}, provided the relation between pseudo-Riemannian metric and Poisson structure by using  the concept of compatibility and introduced the notion of pseudo-Riemannian Poisson manifold. Z. Saassai \cite{zs}, recently shown that the classification of Laplace operator and some other differential operators acting on differential forms.
\par
In \cite{rnm,8yar}, authors explored many explicit formulations for product manifold of Poisson manifolds equipped with product Poisson structure, contravariant warped metric and warped bivector field. Currently in \cite{8bpk}, authors introduced the notions of contravariant Einstein warped product space and Einstein Poisson warped product space.
\par 
The notion of null sectional curvature was introduced by S. G. Harris in \cite{8sgh}.
We know that for pseudo-Riemannian manifolds there exists null vector fields. Moreover, the sectional curvature of the plane spanned  by any of these null vector fields and non null vector field is known as a null sectional curvature.  The null sectional curvature for a plane spanned by null vector and spacelike vector is independent of the choice of spacelike vector. In \cite{8ahb}, authors study the geometrical interpretation of the null sectional curvature in a Lorentzian manifold.  The qualar curvature is defined as the sum of the sectional curvatures of the some plane sections. The notion of qualar curvature was introduced by M. Nardmann  \cite{8mnd}. M. G\"{u}lbahar \cite{8mgl}, in 2020 study the qualar  curvature of pseudo-Riemannian manifolds and pseudo-Riemannian submanifolds. Also, the author  established the relation between qualar curvature and null sectional curvature in his paper.
 \par
 In this paper, we are interested to characterize the explicit form of sectional curvature for two non degenerate independent one forms on contravariant warped product space and Riemannian Poisson warped product space. Additionally, we have investigated Qualar curvatures and null sectional curvatures for the same spaces. In relation with the proved statements, we comment some examples.
\par	
     In Section 2, we recall some basic concepts about contravariant warped product space $(M= M_{1}\times_{f_{1}} M_{2}, \Pi^{\nu_1},\tilde{g}^{f_{1}})$. In Section 3, some properties of $M_{1}$ and $M_{2}$ by using  sectional curvature of $M$ are studied. The Laplacian for a smooth function on $M$ is also obtained.  In Section 4, we introduce  the null, spacelike and timelike forms in dual space of $M$. Then by using these  forms, we find the null sectional curvature and qualar curvature of $M$. In final section, we give examples where the qualar curvatures  for $(M= H_{1}^2\times_{f_{1}} E_{2}^2, \Pi^{\nu_1},\tilde{g}^{f_{1}})$, $(M= H_{1}^2\times_{f_{1}} S_{0}^2, \Pi^{\nu_1},\tilde{g}^{f_{1}})$, and sectional curvatures by using Poisson tensor in $H_{1}^2$, $E_{2}^2$ and $S_{0}^2$ are investigated.
	\section{\textbf{Preliminaries and some results}}
The pair $\left( M_{1}, \lbrace . , .\rbrace\right) $ is known as a Poisson manifold, where $M_{1}$ is a manifold and $\lbrace . , .\rbrace$ is a Poisson bracket on $M_{1}$. For any $\left( M_{1}, \lbrace . , .\rbrace\right)$   there exists a bivector field $\Pi_{1} \in \Gamma\left(\Lambda^2 TM_{1} \right) $, such that 
$$\lbrace f_{1} , f_{2}\rbrace = \Pi_{1}\left( df_{1}, df_{2}\right), \qquad \forall \; f_{1}, f_{2}   \in C^\infty(M_{1}).$$
If Schouten bracket of $\Pi_{1}$ in $\left( M_{1}, \lbrace . , .\rbrace\right)$ is zero then $\Pi_{1}$ is called a Poisson tensor. The Schouten bracket of $\Pi_{1}$ in $\left( M_{1}, \lbrace . , .\rbrace\right)$, is defined as
$$\frac{1}{2} \left[ \Pi_{1}, \Pi_{1}\right]_{S}\left(df_{1}, df_{2}, df_{3} \right)   = \lbrace \lbrace f_{1} , f_{2}\rbrace,  f_{3}\rbrace +  \lbrace \lbrace f_{2} , f_{3}\rbrace,  f_{1}\rbrace  + \lbrace \lbrace f_{3} , f_{1}\rbrace,  f_{2}\rbrace,$$
for all $f_{1}, f_{2}, f_{3} \in C^\infty(M_{1})$. The Poisson manifold  $\left( M_{1}, \lbrace . , .\rbrace\right)$ is also written as $\left( M_{1}, \Pi_{1}\right)$, where $\Pi_{1}$ is a Poisson tensor.

Corresponding to bivector field $\Pi_{1}$ on a manifold $M_{1}$, we can take a natural homomorphism  $\sharp_{\Pi_{1}}: T^* M_{1} \mapsto TM_{1}$,  such that 
$$\eta_{1}\left(\sharp_{\Pi_{1}} (\omega_{1}) \right) =  \Pi_{1}\left(\omega_{1}, \eta_{1} \right), \qquad \forall \; \omega_{1}, \eta_{1} \in \Gamma\left(  T^* M_{1} \right).$$  
The homomorphism $\sharp_{\Pi_{1}}$ is known as sharp map (anchor map) in $M_{1}$.

The Lie bracket on $\Gamma\left(  T^* M_{1} \right)$ of $\left( M_{1}, \Pi_{1}\right)$, is given by
$$\left[ \omega_{1}, \eta_{1} \right]_{\Pi_{1}} = \mathcal{L}_{\sharp_{\Pi_{1}} \left( \omega_{1}\right) } \eta_{1} - \mathcal{L}_{\sharp_{\Pi_{1}} \left( \eta_{1}\right) } \omega_{1} - d\left( \Pi_{1} (\omega_{1}, \eta_{1})\right), \qquad \forall \; \omega_{1}, \eta_{1} \in \Gamma\left(  T^* M_{1} \right).$$

The relation between Lie bracket $\left[ ., . \right]_{\Pi_{1}}$ on $\Gamma\left(  T^* M_{1} \right)$ and Lie bracket $\left[ ., . \right]$ on $\Gamma\left(  T M_{1} \right)$ with the help of anchor map $\sharp_{\Pi_{1}}$, is given by
$$\sharp_{\Pi_{1}} \left( \left[ \omega_{1}, \eta_{1} \right]_{\Pi_{1}}\right) = \left[ \sharp_{\Pi_{1}} (\omega_{1}),  \sharp_{\Pi_{1}} (\eta_{1}) \right], \qquad \forall \; \omega_{1}, \eta_{1} \in \Gamma\left(  T^* M_{1} \right).$$

Let $\left( M_{1}, \bar{g}_{1}\right) $ be a  pseudo-Riemannian manifold of index $q_{1}$ and dimension  $k_{1}$. Then, we can  define a  isomorphism (musical isomorphism) from cotangent  bundle to tangent bundle  of $M_{1}$, such that   
$$b_{\bar{g}_1} : TM_{1} \mapsto T^* M_{1}, \qquad Y_{1} \mapsto \bar{g}_{1}\left( Y_{1}, .\right), \qquad \forall \;Y_{1} \in \Gamma\left(  T M_{1} \right),$$
and its inverse map
$\sharp_{\bar{g}_{1}} : T^* M_{1} \mapsto TM_{1}$,  $\omega_{1} \mapsto \sharp_{\bar{g}_{1}} (\omega_{1})$ such that $\omega_{1} (Y_{1}) = \bar{g}_{1}\left(\sharp_{\bar{g}} (\omega_{1}), Y_{1} \right)$. The cometric  $\tilde{g}_{1}$ of $\bar{g}_{1}$ is defined as 
$$\tilde{g}_{1}(\omega_{1}, \eta_{1} )   =  \bar{g}_{1}\left(\sharp_{\bar{g}_1} (\omega_{1}),  \sharp_{\bar{g}_1} (\eta_{1}) \right).$$

Now, if $\lbrace \frac{\partial}{\partial x^1}, . . . ,  \frac{\partial}{\partial x^{k_{1}}} \rbrace$ and $\lbrace dx_1, . . . , dx_{k_{1}}\rbrace$ are the basis of $\Gamma\left(  TM_{1} \right)$ and $\Gamma\left(  T^* M_{1} \right)$ respectively, then  we can define musical isomorphism  $\flat_{\bar{g}_{1}}$ and its inverse  $\sharp_{\bar{g}_{1}}$, by $\flat_{\bar{g}_{1}} \left(\frac{\partial}{\partial x^i} \right)   = dx^i  $,  $\sharp_{\bar{g}_{1}}\left(  dx_i  \right) =  \frac{\partial}{\partial x^i} $, for all $i \in \lbrace 1, . . . k_{1} \rbrace$, and
\begin{equation}\label{8.2.1}
	\tilde{g}_{1}(dx_i, dx_i )   =  \bar{g}_{1}\left(\sharp_{\bar{g}_{1}} (dx_i),  \sharp_{\bar{g}_{1}} (dx_i) \right) = \bar{g}_{1}\left(\frac{\partial}{\partial x^i},  \frac{\partial}{\partial x^i} \right). 
\end{equation}
Thus from equation (\ref{8.2.1}), we can generalize the idea of non-degenerate, spacelike and timelike vector fields to the duel space   of $\Gamma\left( TM_{1}\right)$. Therefore, a one form $\eta_{1} \in \Gamma\left( T*M_{1}\right)$ is said to be  

(1) Spacelike, if 	$\tilde{g}_{1}(\eta_{1},\eta_{1} )   =  \bar{g}_{1}\left(\sharp_{\bar{g}_1} (\eta_{1}),  \sharp_{\bar{g}_1} (\eta_{1}) \right) > 0,$

(2) Timelike, if $\tilde{g}_{1}(\eta_{1},\eta_{1} )   =  \bar{g}_{1}\left(\sharp_{\bar{g}_1} (\eta_{1}),  \sharp_{\bar{g}_1} (\eta_{1}) \right) < 0,$

(3) Lightlike, if $\tilde{g}_{1}(\eta_{1},\eta_{1} )   =  \bar{g}_{1}\left(\sharp_{\bar{g}_1} (\eta_{1}),  \sharp_{\bar{g}_1} (\eta_{1}) \right) = 0.$\\

For a field endomorphism $J_{1} : T^* M_{1} \rightarrow T^* M_{1}$, the relation between bivector field $\Pi_{1}$ and cometric $\tilde{g}_1$ of $\left( M_{1}, \tilde{g}_{1}\right) $, is given as
$$\Pi_{1} (\omega_{1}, \eta_{1}) = \tilde{g}_{1}\left(J_{1} \omega_{1}, \eta_{1} \right) = - \tilde{g}_{1}\left( \omega_{1},J_{1} \eta_{1} \right), \qquad \forall  \; \omega_{1}, \eta_{1} \in \Gamma\left(  T^* M_{1} \right).  $$

The contravariant derivative of curvature and torsion tensor  with respect to contravariant connection $\mathcal{D}^{M_{1}}$ on Poisson manifold $\left( M_{1}, \Pi_{1}\right)$, are defined by the relations
$$\mathcal{T}^1(\omega_1,\eta_1)=\mathcal{D}^{M_{1}}_{\omega_1}\eta_1-\mathcal{D}^{M_{1}}_{\eta_1}\omega_1-[\omega_{1},\eta_{1}]_{\Pi_{1}},$$
$$\mathcal{R}^1(\omega_1,\eta_1)\gamma_1=\mathcal{D}^{M_{1}}_{\omega_1}\mathcal{D}^{M_{1}}_{\eta_1}\gamma_1-\mathcal{D}^{M_{1}}_{\eta_1}\mathcal{D}^{M_{1}}_{\omega_1}\gamma_1-\mathcal{D}^{M_{1}}_{[{\omega_1},{\eta_1}]_\Pi}\gamma_1,$$
for all $\omega_{1}, \eta_{1}, \gamma_1 \in \Gamma\left(  T^* M_{1} \right)$. There exists a unique contravariant connection (Levi-Civita connection) $\mathcal{D}^{M_{1}}$ on $\left( M_{1}, \Pi_{1}\right)$, with respect to pseudo-Riemannian metric $\tilde{g}$, which is torsion free $\left(\mathcal{T}^1(\omega_1,\eta_1) = 0 \right)$,  and satisfies the condition 
$$\sharp_{\Pi_{1}}(\omega_{1})\tilde{g}_{1}(\eta_{1},\gamma_{1})=\tilde{g}_{1}(\mathcal{D}^{M_{1}}_{\omega_{1}}\eta_{1},\gamma_{1})+\tilde{g}_{1}(\eta_{1},\mathcal{D}^{M_{1}}_{\omega_{1}}\gamma_{1}), \qquad \forall\; \omega_{1}, \eta_{1}, \gamma_{1} \in \Gamma\left(  T^* M_{1} \right).$$

Let $\omega_1,\eta_1$ are linearly independent  covectors of $ T_{p}^* M_{1}$ at $p \in M_{1}$, and $\lbrace dx_1, . . . , dx_{k_{1}}\rbrace$ are local orthonormal basis of $ T_{p}^* M_{1}$. Then sectional curvature $\left(\mathcal{K}_p^1(\omega_1,\eta_1) \right)$, Ricci curvature $\left(Ric_p^1 (\omega_1,\eta_1) \right) $ and scalar curvature $S_{p}^1$ on ${M_{1}}$, will be
\begin{equation}\label{8.2.2}
	\begin{cases}
		\mathcal{K}_p^1(\omega_1,\eta_1)=\frac{\tilde{g}_{{1}_(p)}(\mathcal{R}^1_p(\omega_1,\eta_1)\eta_1,\omega_1)}{\tilde{g}_{{1}_(p)}(\omega_1,\omega_1)\tilde{g}_{{1}_(p)}(\eta_1,\eta_1)-g_p(\omega_1,\eta_1)^2},\\
		Ric^1_p(\omega_1,\eta_1)=\displaystyle\sum_{j=1}^{k_{1}}\tilde{g}_{{1}_(p)}(\mathcal{R}^1_p(\omega_1,dx_j)dx_j,\eta_1),\\
		S^1_p=\displaystyle\sum_{i=1}^{k_{1}}Ric^1_p(dx_i,dx_i).
	\end{cases}
\end{equation}

A function $f_{k}   \in C^\infty(M_{1})$ is said to be a Casimir function if 
$$\lbrace f_{k}, f_{i}\rbrace = 0\qquad \forall \; f_{i}   \in C^\infty(M_{1}).$$

The function $f_{k}   \in C^\infty(M_{1})$, will be a Casimir function if and only if $J_{1}df_{k} = 0$.
\begin{definition}
	Let $\tilde {g}^{f_1}$ be a cometric  of the warped metric $\bar{g}^{f_{1}}=\pi_{1}^*(\bar{g}_1)+(f_{1}^h)^2\pi_{2}^*(\bar{g}_2)$, on product manifold $M = M_{1} \times M_{2}$, where $f_{1}^h=f_{1}\circ\pi_{1}$ is the horizontal lift of a positive smooth function $f_{1}$ from $M_{1}$ to $M$. Then the ordered pair $(M= M_{1}\times_{f_1}  M_{2} , \tilde{g}^{f_{1}})$ is said to be a contravariant warped product space of the warped product space $(M= M_{1}\times_{f_1}  M_{2} , \bar{g}^{f_{1}})$.
\end{definition}

 \subsection{Contravariant warped product space}
Let $\left( M_{1}, \Pi_{1}\right)$ and $\left( M_{2}, \Pi_{2}\right)$ are Poisson manifolds equipped with semi-Riemannian metric (cometric) $\bar{g}_{1}$  ($\tilde{g}_1$) and $\bar{g}_{2}$  ($\tilde{g}_2$), respectively. Then if  $\mathcal{D}^{M_{1}}$ and $\mathcal{D}^{M_{2}}$ are the contravariant Levi-Civita connections associated with pairs $(\tilde{g}_1,\Pi_1)$ and $(\tilde{g}_2,\Pi_2)$, we can associate a Levi-Civita connections $\mathcal{D}$ with the pair $(\tilde{g}^{f_{1}},\Pi^{\nu_{1}})$ on contravariant  warped product space $(M= M_{1}\times_{f_{1}} M_{2}, \tilde{g}^{f_{1}})$, where $\tilde{g}^{f_{1}}=\tilde{g}_1^h+\frac{1}{({f_{1}}^h)^2}\tilde{g}_2^v$, $\Pi^{\nu_{1}}=\Pi_1+\nu_{1}\Pi_2$, $\nu_{1} \in C^{\infty} M_{1}$ and ${f_{1}}$ be a positive smooth function on $M_{1}$. The lifts of smooth functions, vectors, covectors, metric and cometric from $\left( M_{1}, \Pi_{1}, \tilde{g}_1 \right)$, and $\left( M_{2}, \Pi_{2}, \tilde{g}_2\right)$  to $(M= M_{1}\times_{f_{1}} M_{2}, \Pi^{\nu_1},\tilde{g}^{f_{1}})$ corresponding to projection maps $\pi_{1}:M \rightarrow M_{1}$ and $\pi_{2}: M \rightarrow M_{2}$, are given by 

(1) If ${f_{1}}, \nu_{1} \in C^{\infty} M_{1}$, then the horizontal lifts of $\left({f_{1}}, \nu_{1} \right)$ to $M$ is  $f_{1} \circ \pi_{1} = f_{1}^h , \nu_{1} \circ \pi_{1} = \nu_{1}^h \in C^{\infty} M$.

(2) Let $X_{1}$ be a smooth section of $TM_{1}$ then the horizontal lift of $X_{1}$ to $M$ is a vector field $X_{1}^h$ in $TM$, whose value at point $(p_{1},p_{2}) \in M$ ($p_{i} \in M_{i}$ for $i = 1,2$), is the horizontal lift of tangent vector $(X_{1})_{p} \in \Gamma (TM_{1})$ to $(p_{1},p_{2})$. Similarly we can vertically lift the smooth section of $TM_{2}$ to $TM$.

Now, as $d\pi_{1}:TM \rightarrow TM_{1}$ and $d\pi_{2}: TM \rightarrow TM_{2}$, therefore for horizontal and vertical lifts of vector fields, we have
$$d_{(p_{1},p_{2})}\pi_{1}(X_{(p_{1},p_{2})}^h)=X_{p_1}, \qquad d_{(p_{1},p_{2})}\pi_{2}(X_{(p_{1},p_{2})}^h)=0, $$
$$d_{(p_{1},p_{2})}\pi_{1}(Y_{(p_{1},p_{2})}^v)= 0, \qquad d_{(p_{1},p_{2})}\pi_{2}(Y_{(p_{1},p_{2})}^v)=Y_{p_{2}} .$$

(3) Let $\alpha_{1}$ be a smooth section of $T^*M_{1}$, then the horizontal lift of $\alpha_{1}$ to $M$ is $\pi_{1}^*(\alpha_{1})=\alpha_{1}^h$, such that $(\alpha_1^h)(X)=\alpha_1(d\pi_{1}(X)), \forall\; X\in\Gamma(T(M_{1}\times M_{2}))$. Similarly, for vertical lift of smooth section $\alpha_{2}$ of $T^*M_{2}$,  $\pi_{2}^*(\alpha_{2})=\alpha_{2}^v \in T^*M$.\\
\par
Let $\bar{g}^{f_1}=\pi_{1}^*(\bar{g}_1)+(f_{1}^h)^2\pi_{2}^*(\bar{g}_2)$ be a warped metric and $\tilde {g}^{f_1} = \tilde{g}_{1}^h+\frac{1}{(f_{1}^h)^2}\tilde{g}_2^v$ be its  cometric. Then for the horizontal and vertical lifts of vector fields and one forms are given by respectively 
\begin{equation*}
	\begin{cases}
	
		\bar{g}^{f_1}(A_1^h,B_1^h)=\bar{g}_1(A_1,B_1)^h,\\
		\bar{g}^{f_1}(A_1^h,B_2^v)=\bar{g}_1(A_2^v,B_1^h)=0,\\
		\bar{g}^{f_1}(A_2^v,B_2^v)=(f_{1}^h)^2\bar{g}_2(A_2,B_2)^v,
	\end{cases}	
\end{equation*}
where $A_{1}, B_{1} \in\Gamma(TM_{1})$ and $A_{2}, B_{2} \in\Gamma(TM_{2})$. And
\begin{eqnarray}\label{8.2.3}
	\left\{
	\begin{array}{ll}
	\tilde{g}^{f_1}(\omega_1^h,\gamma_1^h)=\tilde{g}_1(\omega_1,\gamma_1)^h,\\
	\tilde{g}^{f_1}(\omega_1^h,\gamma_2^v)=\tilde{g}_1(\omega_2^v,\gamma_1^h)=0,\\
	\tilde{g}^{f_1}(\omega_2^v,\gamma_2^v)=\frac{1}{(f_{1}^h)^2}\tilde{g}_2(\omega_2,\gamma_2)^v.
	\end{array}
	\right.
\end{eqnarray}
where $\omega_{1}, \gamma_{1} \in\Gamma(T^*M_{1})$ and $\omega_{2}, \gamma_{2} \in\Gamma(T^*M_{2})$.\\
\par
Next we are extracting few results from \cite{8bpk, 8yar}, which are  used  in our paper  

\begin{proposition}\cite{8yar}
	The contravariant Hessian $H_{\Pi}^\varphi$ of $(0,2)$-type tensor field  $\varphi$ on $\left( M_{1}, \Pi_{1}\right) $, satisfies the relation
	\begin{equation*}
		H_{\Pi_{1}}^\varphi(\omega_{1},\eta_{1})=\sharp_{\Pi_{1}}(\omega_{1})(\sharp_{\Pi_{1}}(\eta_{1})(\varphi))-\sharp_{\Pi_{1}}(\mathcal{D}^{M_1}_{\omega_{1}}\eta_{1})(\varphi)=-g(\mathcal{D}^{M_1}_{\omega_{1}}J_{1}d\varphi,\eta_{1}).
	\end{equation*}
	Moreover, for Poisson tensor $\Pi_{1}$,  $H_{\Pi_1}^\varphi$ is symmetric.
\end{proposition}
\begin{proposition}\cite{8yar} \label{p2.3}
	Let $\omega_1,\gamma_1\in\Gamma(T^*M_{1})$, $\omega_2,\gamma_2\in\Gamma(T^*M_{2})$ and $\omega=\omega_1^h+\omega_2^v$,  $\gamma=\gamma_1^h+\gamma_2^v$ are one forms on product space $\left( M = M_{1} \times M_{2}, \Pi^{\nu_{1}}\right)$, then
	\begin{align*}
		&(1).\:\sharp_{\Pi^{\nu_1}}(\omega)=\big[\sharp_{\Pi_1}(\omega_1)\big]^h+{{\nu_{1}}}^h\big[\sharp_{\Pi_2}(\omega_2)\big]^v,\\
		&(2).\:\mathcal{L}_{\sharp_{\Pi^{\nu_1}}(\omega)}\gamma=\big(\mathcal{L}_{\sharp_{\Pi_1}(\omega_1)}\gamma_1\big)^h+\nu_{1}^h\big(\mathcal{L}_{\sharp_{\Pi_2}(\omega_2)}\gamma_2\big)^v+\Pi_2(\omega_2,\gamma_2)^v(d\nu_{1})^h,\\
		&(3).\:[\omega,\gamma]_{\Pi^{\nu_{1}}}=[\omega_1,\gamma_1]_{\Pi_1}^h+\nu_{1}^h[\omega_2,\gamma_2]_{\Pi_2}^v+\Pi_2(\omega_2,\gamma_2)^v(d\nu_{1})^h.
	\end{align*}
\end{proposition}
Next, the following two propositions are discussed for the contravariant warped product space $(M= M_{1}\times_{f_{1}} M_{2}, \Pi^{\nu_1},\tilde{g}^{f_{1}})$. 
	\begin{proposition}\cite{8bpk} \label{p2.4}
	Let $\omega_1,\gamma_1\in\Gamma(T^*M_{1})$ and $\omega_2,\gamma_2\in\Gamma(T^*M_{2})$, then
	\begin{align*}
		(1).\:\mathcal{D}_{\omega_1^h}\gamma_1^h&=(\mathcal{D}_{\omega_1}^{M_{1}}\gamma_1)^h,\\
		(2).\:\mathcal{D}_{\omega_2^v}\gamma_2^v&={\nu_1}^h(\mathcal{D}_{\omega_2}^{M_{2}}\gamma_2)^v+\frac{1}{2}\Pi_2(\omega_2,\gamma_2)^v(d{\nu_1})^h-\frac{1}{(f_{1}^h)^3}\tilde{g}_2(\omega_2,\gamma_2)^v(J_1df_{1})^h,\\
		(3).\:\mathcal{D}_{\omega_1^h}\gamma_2^v&=\mathcal{D}_{\gamma_2^v}\omega_1^h=\frac{1}{2f_{1}^h}\big[2\tilde{g}_1(J_1df_{1},\omega_1)^h\gamma_2^v-\{f_{1}^3\tilde{g}_1(d{\nu_1},\omega_1)\}^h(J_2\gamma_2)^v\big].
	\end{align*}
\end{proposition}
\begin{proposition}\cite{8bpk} \label{p8.2.5}
Consider the one forms	 $\omega_1,\eta_1,\gamma_1\in\Gamma(T^*M_{1})$, $\omega_2,\eta_2,\gamma_2\in\Gamma(T^*M_{2})$ and $\gamma=\gamma_1^{h}+\gamma_2^{v}$, then
	\begin{align*}
	(1).\:\mathcal{R}(&\omega_1^{h},\eta_1^{h})\gamma=\big[\mathcal{R}_1(\omega_1,\eta_1)\gamma_1\big]^h\\
	&+\frac{1}{f_{1}^h}\big[\tilde{g}_1(\mathcal{D}_{\omega_1}^{M_{1}}J_1df_{1},\eta_1)-\tilde{g}_1(\mathcal{D}_{\eta_1}^{M_{1}}J_1df_{1},\omega_1)\big]^h\gamma_2^v\\
	&+\frac{1}{(f_{1}^h)^2}\big[\mathcal{D}_{\eta_1}^{M_{1}}(f_{1})\tilde{g}_1(J_1df_{1},\omega_1)-\mathcal{D}_{\omega_1}^{M_{1}}(f_{1})\tilde{g}_1(J_1df_{1},\eta_1)\big]^h\gamma_2^v\\
	&+\frac{(f_{1}^h)^2}{2}\big[\tilde{g}_1(\mathcal{D}_{\eta_1}^{M_{1}}d{\nu_{1}},\omega_1)-\tilde{g}_1(\mathcal{D}_{\omega_1}^{M_{1}}d{\nu_{1}},\eta_1)\big]^h(J_2\gamma_2)^v\\
	&+f_{1}^h\big[\mathcal{D}_{\eta_1}^{M_{1}}(f_{1})\tilde{g}_1(d{\nu_{1}},\omega_1)-\mathcal{D}_{\omega_1}^{M_{1}}(f_{1})\tilde{g}_1(d{\nu_{1}},\eta_1)\big]^h(J_2\gamma_2)^v,
	\end{align*}
	\begin{align*}
	(2).\:\mathcal{R}(&\omega_1^{h},\eta_2^{v})\gamma_1^h=\frac{1}{(f_{1}^h)^{2}}\big[\tilde{g}_1(J_1df_{1},\omega_1)\tilde{g}_1(J_1df_{1},\gamma_1)\big]^h\eta_2^{v}+\tilde{g}_1(\mathcal{D}_{\omega_1}^{M_{1}}(\frac{J_1df_{1}}{f_{1}}),\gamma_1)^h\eta_2^{v}\\
	&-\frac{f_{1}^h}{2}\big[\tilde{g}_1(d{\nu_{1}},\omega_1)\tilde{g}_1(J_1df_{1},\gamma_1)+\tilde{g}_1(J_1df_{1},\omega_1)\tilde{g}_1(d{\nu_{1}},\gamma_1)\big]^h(J_2\eta_2)^v\\
	&-\tilde{g}_1(\mathcal{D}_{\omega_1}^{M_{1}}\frac{f_{1}^2d{\nu_{1}}}{2},\gamma_1)^h(J_2\eta_2)^v+\frac{(f_{1}^h)^4}{4}\big[\tilde{g}_1(d{\nu_{1}},\omega_1)\tilde{g}_1(d{\nu_{1}},\gamma_1)\big]^h(J_2^2\eta_2)^v,
	\end{align*}
	\begin{align*}
	(3).\:\mathcal{R}(&\omega_1^{h},\eta_2^{v})\gamma_2^v=-\frac{1}{(f_{1}^h)^3}\tilde{g}_2(\eta_2,\gamma_2)^v(\mathcal{D}_{\omega_1}^{M_{1}}J_1df_{1})^h-\Pi_1(d{\nu_{1}},\omega_1)^h(\mathcal{D}_{\eta_2}^{M_{2}}\gamma_2)^v\\
	&-\frac{1}{2(f_{1}^h)^4}\big[(f_{1}^3\tilde{g}_1(d{\nu_{1}},\omega_1))^h\tilde{g}_2(J_2\gamma_2,\eta_2)^v+4\tilde{g}_1(J_1df_{1},\omega_1)^h\tilde{g}_2(\eta_2,\gamma_2)^v\big](J_1df_{1})^h\\
	&+\frac{1}{2}\big[({\nu_{1}} f_{1}^2\tilde{g}_1(d{\nu_{1}},\omega_1))^h\{\mathcal{D}_{\eta_2}^{M_{2}}J_2\gamma_2-J_2\mathcal{D}_{\eta_2}^{M_{2}}\gamma_2\}^v+\Pi_2(\eta_2,\gamma_2)^v(\mathcal{D}_{\omega_1}^{M_{1}}d{\nu_{1}})^h\big]\\
	&+\frac{1}{4f_{1}^h}\big[(f_{1}^3\tilde{g}_1(d{\nu_{1}},\omega_1))^h\Pi_2(\eta_2,J_2\gamma_2)^v-2\tilde{g}_1(J_1df_{1},\omega_1)^h\Pi_2(\eta_2,\gamma_2)^v\big](d{\nu_{1}})^h,
	\end{align*}
	\begin{align*}
	(4).\:\mathcal{R}(&\omega_2^{v},\eta_2^{v})\gamma_1^h=\frac{1}{f_{1}^h}\Pi_2(\omega_2,\eta_2)^v\big[\tilde{g}_1(J_1df_{1},\gamma_1)d{\nu_{1}}-\tilde{g}_1(d{\nu_{1}},\gamma_1)(J_1df_{1})-f_{1}\mathcal{D}_{d{\nu_{1}}}^{M_{1}}\gamma_1\big]^h\\
	&+\Big(\frac{f_{1}^2{\nu_{1}} \tilde{g}_1(d{\nu_{1}},\gamma_1)}{2}\Big)^h\big[(\mathcal{D}_{\eta_2}^{M_{2}}(J_2\omega_2)-(\mathcal{D}_{\omega_2}^{M_{2}}(J_2\eta_2)+J_2[\omega_2,\eta_2]_{\Pi_{2}}\big]^v,
\end{align*}
\begin{align*}
	(5).\:\mathcal{R}(\omega_2^{v},\eta_2^{v}&)\gamma_2^v=({\nu_{1}}^h)^2\big[\mathcal{R}_2(\omega_2,\eta_2)\gamma_2\big]^v\\
	&+\frac{{\nu_{1}}^h}{2}\big[(\mathcal{D}_{\omega_2}^{M_{2}}\Pi_{2})(\eta_2,\gamma_2)-(\mathcal{D}_{\eta_2}^{M_{2}}\Pi_{2})(\omega_2,\gamma_2)\big]^h(d{\nu_{1}})^h\\
	&+\Big(\frac{f_{1}^2||d{\nu_{1}}||_1^{2}}{4}\Big)^h\big[J_2\{\Pi_2(\omega_2,\gamma_2)\eta_2-\Pi_2(\eta_2,\gamma_2)\omega_2+2\Pi_2(\omega_2,\eta_2)\gamma_2\}\big]^v\\
	&+\Big(\frac{||J_1df_{1}||_1^{2}}{f_{1}^4}\Big)^h\big[\tilde{g}_2(\omega_2,\gamma_2)\eta_2-\tilde{g}_2(\eta_2,\gamma_2)\omega_2\big]^v\\
	&+\Big(\frac{\tilde{g}_1(d{\nu_{1}},J_1df_{1})}{2f_{1}}\Big)^h\big[\Pi_2(\eta_2,\gamma_2)\omega_2-\Pi_2(\omega_2,\gamma_2)\eta_2-2\Pi_2(\omega_2,\eta_2)\gamma_2\big]^v\\
	&+\Big(\frac{\tilde{g}_1(d{\nu_{1}},J_1df_{1})}{2f_{1}}\Big)^h\big[J_2\{\tilde{g}_2(\eta_2,\gamma_2)\omega_2-\tilde{g}_2(\omega_2,\gamma_2)\eta_2\}\big]^v.
\end{align*}
\end{proposition}
\section{\textbf{ Sectional curvature}}
In this section, we find the sectional curvature of contravariant warped product space $(M= M_{1}\times_{f_{1}} M_{2}, \Pi^{\nu_1},\tilde{g}^{f_{1}})$. We also established the relation between  sectional curvatures of  $M_{1}, M_{2}$ with the sectional curvature of $M$ by taking $f_{1}$ as a Casimir function and $\nu_{1} = constant$.

\begin{lemma}\label{lem3.1}
	Let $(M= M_{1}\times_{f_{1}} M_{2}, \Pi^{\nu_1},\tilde{g}^{f_{1}})$ be a contravariant  warped product space. Then for non degenerate independent one forms   $\omega_1,\eta_1\in\Gamma(T^*M_{1})$ and $\omega_2,\eta_2\in\Gamma(T^*M_{2})$, we have 
\begin{align*}
	&\mathcal{K}^M\left(\omega_{1}^h,\eta_{1}^h \right)  =  \mathcal{K}^{M_{1}}\left( \omega_{1},\eta_{1}\right) ^h, 
	\end{align*}
\begin{align*}	
 \mathcal{K}^M\left(\omega_{1}^h,\eta_{2}^v \right)& = - \left(\frac{\tilde{g}_{1}\left( {\mathcal{D}^{M_{1}}_{\omega_{1}}J_{1} df_{1}}, \omega_{1} \right) }{f_{1} \tilde{g}_{1}\left(\omega_{1}, \omega_{1} \right) } \right)^h  -  \frac{2 \left( |\tilde{g}_{1}\left(J_{1}df_{1}, \omega_{1} \right)^h |^2\right) }{(f_{1 }^h)^2\tilde{g}_{1}\left(\omega_{1}, \omega_{1} \right)}\\
&+ \frac{\tilde{g}_{2}\left( J_{2}\eta_{2}, J_{2}\eta_{2}\right)^v |\left( f_{1}^2 \tilde{g}_{1}\left( d\nu_{1}, \omega_{1}\right) \right)^h|^2 }{4\tilde{g}_{1}\left(\omega_{1}, \omega_{1} \right)^h \tilde{g}_{2}\left(\eta_{2}, \eta_{2} \right)^v },
\end{align*}
\begin{align*}	
 \mathcal{K}^M\left(\omega_{2}^v,\eta_{2}^v \right)& = (\nu_{1}^h)^2 (f_{1}^h)^2  \mathcal{K}^{M_{2}}\left(\omega_{2},\eta_{2}\right)^v - \frac{\left(\||J_{1}df_{1}||_{1}^2 \right)^h}{(f_{1}^h)^2} \\
 &- \frac{\left( 3f_{1}^4 ||d\nu_{1}||_{1}^2\right)^h |g_{2}\left( J_{2} \omega_{2} , \eta_{2}\right)^v|^2 + 4 \left( f_{1} \tilde{g}_{1}\left(d\nu_{1}, J_{1}df \right) \right)^h \tilde{g}_{2}\left(\omega_{2} , \eta_{2}\right)^v \tilde{g}_{2}\left(J_{2} \omega_{2} , \eta_{2}\right)^v}{4\left( \tilde{g}_{2}\left( \omega_{2}, \omega_{2} \right)^v \tilde{g}_{2}\left( \eta_{2}, \eta_{2} \right)^v - |\tilde{g}_{2}\left( \omega_{2}, \eta_{2} \right)^v|^2 \right) },
\end{align*}

where $\mathcal{K}^M$, $\mathcal{K}^{M_{1}}$ and $\mathcal{K}^{M_{2}}$ are sectional curvatures of $M$, $M_{1}$ and $M_{2}$ respectively.
\end{lemma}
\begin{proof}
	Let $(M= M_{1}\times_{f_{1}} M_{2}, \Pi^{\nu_1},\tilde{g}^{f_{1}})$ be a contravariant warped product space with non degenerate independent one forms   $\omega_1,\eta_1\in\Gamma(T^*M_{1})$ and $\omega_2,\eta_2\in\Gamma(T^*M_{2})$ . Then sectional curvature of the plane spanned by $\lbrace \omega_{1}^h, \eta_{1}^h \rbrace$ in $\Gamma\left( T^* M\right)$
	$$\mathcal{K}^M(\omega_{1}^h,\eta_{1}^h)=\frac{\tilde{g}^{f_{1}}(\mathcal{R}(\omega_{1}^h,\eta_{1}^h)\eta_{1}^h,\omega_{1}^h)}{\tilde{g}^{f_{1}}(\omega_{1}^h,\omega_{1}^h)\tilde{g}^{f_{1}}(\eta_{1}^h,\eta_{1}^h)-\tilde{g}^{f_{1}}(\omega_{1}^h,\eta_{1}^h)^2}.$$ Since, $$\tilde{g}^{f_{1}}(\mathcal{R}(\omega_{1}^h,\eta_{1}^h)\eta_{1}^h,\omega_{1}^h) = \tilde{g}^{f_{1}} \left( \big[\mathcal{R}_1(\omega_1,\eta_1),\eta_{1}, \big]^h,\omega_{1}^h \right) = g_{1}\left( \mathcal{R}_1(\omega_1,\eta_1),\eta_{1}, \omega_{1}\right)^h $$
	and,
$${\tilde{g}^{f_{1}}(\omega_{1}^h,\omega_{1}^h)\tilde{g}^{f_{1}}(\eta_{1}^h,\eta_{1}^h)-\tilde{g}^{f_{1}}(\omega_{1}^h,\eta_{1}^h)^2} = \tilde{g}_1(\omega_{1},\omega_{1})^h\tilde{g}_{1}(\eta_{1},\eta_{1})^h-|\tilde{g}_{1} (\omega_{1},\eta_{1})^h| ^2.$$
 Hence, $\mathcal{K}^M\left(\omega_{1}^h,\eta_{1}^h \right)  =  \mathcal{K}^{M_{1}}\left( \omega_{1},\eta_{1}\right) ^h,$	where 
$$ \mathcal{K}^{M_{1}}\left( \omega_{1},\eta_{1}\right) ^h =  \left( \frac{\tilde{g}_{1}(\mathcal{R}_{1}(\omega_{1},\eta_{1})\eta_{1},\omega_{1})}{\tilde{g}_{1}(\omega_{1},\omega_{1})\tilde{g}_{1}(\eta_{1},\eta_{1})-|\tilde{g}_{1}(\omega_{1},\eta_{1})|^2}\right)^h.$$
 Now, for a plane spanned by $\lbrace \omega_{1}^h, \eta_{2}^v \rbrace$ in $\Gamma\left( T^* M\right)$, the sectional curvature
 \begin{equation}\label{8.3.1}
 	\mathcal{K}^M(\omega_{1}^h,\eta_{2}^v)=\frac{\tilde{g}^{f_{1}}(\mathcal{R}(\omega_{1}^h,\eta_{2}^v)\eta_{2}^v,\omega_{1}^h)}{\tilde{g}^{f_{1}}(\omega_{1}^h,\omega_{1}^h)\tilde{g}^{f_{1}}(\eta_{2}^v,\eta_{2}^v)}.
 \end{equation}
  
 Using third part of Proposition \ref*{p8.2.5}, we have
 \begin{align}\label{8.3.2}
 	\nonumber	\tilde{g}^{f_{1}}(\mathcal{R}(\omega_{1}^h,\eta_{2}^v)\eta_{2}^v,\omega_{1}^h) &= -\left(\frac{1}{f_{1}^h} \right)^3 \tilde{g}_{2} \left( \eta_{2}, \eta_{2}\right)^v \tilde{g}_{1} \left(\mathcal{D}_{\omega_{1}}^{M_{1}} J_{1}df_{1}, \omega_{1} \right)^h, \\
  \nonumber	&- \frac{2}{(f_{1}^h)^4} |\tilde{g}_{1}\left(J_{1}df_{1}, \omega_{1} \right)^2|^2 \tilde{g}_{2} \left(\eta_{2}, \eta_{2} \right)^v,\\
	 &+ \left( \frac{f_{1}^h}{2}  \right)^2 |\tilde{g}_{1}\left(d\nu_{1}, \omega_{1} \right)^h|^2 \tilde{g}_{2}\left( J_{2}\eta_{2}, J_{2}\eta_{2} \right).  
 \end{align}
Also, 
\begin{align}\label{8.3.3}
		\tilde{g}^{f_{1}}(\omega_{1}^h,\omega_{1}^h)\tilde{g}^{f_{1}}(\eta_{2}^v,\eta_{2}^v)&=
	\left( \frac{1}{{f_{1}}^h}\right)^2 \tilde{g}_{1}(\omega_{1},\omega_{1})^h\tilde{g}_{2}(\eta_{2},\eta_{2})^v. 
\end{align}
Hence from equations  (\ref{8.3.1}), (\ref{8.3.2}) and (\ref{8.3.3}), we get the second part of our lemma. Similarly, third part of the lemma can be proved by  using last expression of Proposition \ref{p8.2.5}, and the relation 
\begin{equation*}
	\mathcal{K}^M(\omega_{2}^v,\eta_{2}^v)=\frac{\tilde{g}^{f_{1}}(\mathcal{R}(\omega_{2}^v,\eta_{2}^v)\eta_{2}^v,\omega_{2}^v)}{\tilde{g}^{f_{1}}(\omega_{2}^v,\omega_{2}^v)\tilde{g}^{f_{1}}(\eta_{2}^v,\eta_{2}^v)	-|\tilde{g}^{f_{1}}(\eta_{2}^v,\omega_{2}^v)|^2}.
\end{equation*}  
\end{proof}
 \begin{theorem}
 	Let $f_{1}$ be a Casimir function and $(M= M_{1}\times_{f_{1}} M_{2}, \Pi^{\nu_1},\tilde{g}^{f_{1}})$ be a Riemannian Poisson warped product space. Then positive sectional curvature of $M$ implies that $M_{1}$ and $M_{2}$ are  Poisson  manifolds of positive sectional curvatures.
 \end{theorem}
\begin{proof}
		Let $f_{1}$ be a Casimir function and $(M= M_{1}\times_{f_{1}} M_{2}, \Pi^{\nu_1},\tilde{g}^{f_{1}})$ be a Riemannian Poisson warped product space. Then using $J_{1}df_{1} = 0$ in Lemma \ref*{lem3.1}, we have
		\begin{align}
		\label{8.3.4}	&\mathcal{K}^M\left(\omega_{1}^h,\eta_{1}^h \right)  =  \mathcal{K}^{M_{1}}\left( \omega_{1},\eta_{1}\right) ^h, \\		
		\label{8.3.5}	\mathcal{K}^M\left(\omega_{1}^h,\eta_{2}^v \right)& =
			 \frac{\tilde{g}_{2}\left( J_{2}\eta_{2}, J_{2}\eta_{2}\right)^v |\left( f_{1}^2 \tilde{g}_{1}\left( d\nu_{1}, \omega_{1}\right) \right)^h|^2 }{4\tilde{g}_{1}\left(\omega_{1}, \omega_{1} \right)^h \tilde{g}_{2}\left(\eta_{2}, \eta_{2} \right)^v },\\
		\label{8.3.6}	\mathcal{K}^M\left(\omega_{2}^v,\eta_{2}^v \right)& = (\nu_{1}^h)^2 (f_{1}^h)^2  \mathcal{K}^{M_{2}}\left(\omega_{2},\eta_{2}\right)^v 
			- \frac{\left( 3f_{1}^6 ||d\nu_{1}||_{1}^2\right)^h |\tilde{g}_{2}\left( J_{2} \omega_{2} , \eta_{2}\right)^v|^2 }{\tilde{g}_{2}\left( \omega_{2}, \omega_{2} \right)^v \tilde{g}_{2}\left( \eta_{2}, \eta_{2} \right)^v - |\tilde{g}_{2}\left( \omega_{2}, \eta_{2} \right)^v|^2 }.
		\end{align}
	If $\mathcal{K}^M > 0$, then from (\ref{8.3.4}), $\mathcal{K}^{M_{1}} > 0$. Then from equation (\ref{8.3.6}), $\mathcal{K}^M > 0$ implies that
	$$(\nu_{1}^h)^2 (f_{1}^h)^2  \mathcal{K}^{M_{2}}\left(\omega_{2},\eta_{2}\right)^v 
	> \frac{\left( 3f_{1}^6 ||d\nu_{1}||_{1}^2\right)^h |\tilde{g}_{2}\left( J_{2} \omega_{2} , \eta_{2}\right)^v|^2 }{\tilde{g}_{2}\left( \omega_{2}, \omega_{2} \right)^v \tilde{g}_{2}\left( \eta_{2}, \eta_{2} \right)^v - |\tilde{g}_{2}\left( \omega_{2}, \eta_{2} \right)^v|^2 } >0.$$
\end{proof}                                                                                \begin{corollary}\label{c3.3}
Let $(M= M_{1}\times_{f_{1}} M_{2}, \Pi^{\nu_1},\tilde{g}^{f_{1}})$ be a Riemannian Poisson warped product space with $\nu_{1} = constant$ and  $f_{1}$ be a Casimir function on  $M_{1}$. Then $M$  have non negative (non positive) sectional curvature if and only if $M_{1}$ and $M_{2}$ are  Poisson  manifolds of non negative (non positive) sectional curvatures.
\end{corollary}     
\begin{proof}
	Proof for Corollary, directly followed from Theorem 3.2, by taking $d\nu_{1} = 0$ in equations (\ref{8.3.4}), (\ref{8.3.5}) and (\ref{8.3.6}).
\end{proof}
\begin{theorem}\label{t3.4}
	Let $M_{1}$ be a Riemannian manifold and $(M= M_{1}\times_{f_{1}} M_{2}, \Pi^{\nu_1},\tilde{g}^{f_{1}})$ be a contravariant warped product space with $\nu_{1} = constant$. Then $(M= M_{1}\times_{f_{1}} M_{2}, \Pi^{\nu_1},\tilde{g}^{f_{1}})$,  will have non-negative sectional curvature if and only if $M_{1}$, $M_{2}$ have non negative sectional curvatures and $f_{1}$ be a Casimir function.
\end{theorem}
\begin{proof}
		Let $M_{1}$ be a Riemannian manifold and $(M= M_{1}\times_{f_{1}} M_{2}, \Pi^{\nu_1},\tilde{g}^{f_{1}})$ be a contravariant Poisson warped product space with $\nu_{1} = constant$. Then taking $d\nu_{1} = 0$ in Lemma 3.1, we have
		\begin{align}
			\label{8.3.7}&\mathcal{K}^M\left(\omega_{1}^h,\eta_{1}^h \right)  =  \mathcal{K}^{M_{1}}\left( \omega_{1},\eta_{1}\right) ^h, \\
			\label{8.3.8} \mathcal{K}^M\left(\omega_{1}^h,\eta_{2}^v \right)& = - \left(\frac{\tilde{g}_{1}\left( {\mathcal{D}^{M_{1}}_{\omega_{1}}J_{1} df_{1}}, \omega_{1} \right) }{f_{1} \tilde{g}_{1}\left(\omega_{1}, \omega_{1} \right) } \right)^h  - 2  \frac{|\tilde{g}_{1}\left(J_{1}df_{1}, \omega_{1} \right)^h |^2}{\tilde{g}_{1}\left(\omega_{1}, \omega_{1} \right)},\\	
		\label{8.3.9}	\mathcal{K}^M\left(\omega_{2}^v,\eta_{2}^v \right)& = (\nu_{1}^h)^2 (f_{1}^h)^2  \mathcal{K}^{M_{2}}\left(\omega_{2},\eta_{2}\right)^v - \left(\||J_{1}df||_{1}^2 \right)^h. 
		\end{align}
	Now, for $\mathcal{K}^M \geq 0$, using the fact that $M_{1}$ is a Riemannian manifold in equation (\ref{8.3.8}), we get
	$$ \tilde{g}_{1}\left( \mathcal{D}_{\omega_{1}}^{M_{1}} J_{1}df_{1}, \omega_{1}\right) = 0 = \tilde{g}_{1}\left(J_{1} df_{1}, \omega_{1}\right), \qquad \forall \;\omega_{1} \in \Gamma\left(TM_{1} \right).$$
	\begin{equation}\label{8.3.10}
		\implies J_{1} df_{1} = 0.
	\end{equation}
Using (\ref{8.3.10}), in (\ref{8.3.9}), we obtain 
\begin{equation}\label{8.3.11}
	\mathcal{K}^M\left(\omega_{2}^v,\eta_{2}^v \right) = (\nu_{1}^h)^2 (f_{1}^h)^2  \mathcal{K}^{M_{2}}\left(\omega_{2},\eta_{2}\right)^v.
\end{equation}
Hence, equations (\ref{8.3.7}), (\ref{8.3.10}) and (\ref{8.3.11}) implies that $M_{1}$ and $M_{2}$ are manifolds of non-negative sectional curvatures. 

Converse of the theorem directly followed from the Corollary \ref{c3.3} .
\end{proof}   
\begin{corollary}\label{c3.5}
	Let $(M= \left(\mathbb{R}^n\; or S^n\; \right) \times_{f_{1}} M_{2}, \Pi^{\nu_1},\tilde{g}^{f_{1}})$, $n\geq 2$,  $\nu_{1} = constant$,  be a contravariant warped product space. Then $M$  have non negative  sectional curvature if and only if  $M_{2}$ be a  Poisson  manifold of non negative sectional curvature.
\end{corollary}     
\begin{proof}
	Because $\mathbb{R}^n$ and $S^n$, $n\geq 2$ are Riemannian manifolds therefore from Theorem \ref{t3.4}, $J_{1}df = 0$, and 
	\begin{equation*}
		\mathcal{K}^M\left(\omega_{2}^v,\eta_{2}^v \right) = (\nu_{1}^h)^2 (f_{1}^h)^2  \mathcal{K}^{M_{2}}\left(\omega_{2},\eta_{2}\right)^v.
	\end{equation*}
Thus, completes the proof.
\end{proof}
\subsection{Laplacian for a smooth function on the contravariant warped product space} 
Let $\lbrace dx^h_{1}, . . . dx^h_{k_{1}}, {f_{1}}^h dy_{1}^v, . . .  {f_{1}}^h dy_{k_{2}}^v \rbrace$ be local $\tilde{g}^{f_1}$-  orthonormal basis on contravariant warped Product space $(M= M_{1}\times_{f_{1}} M_{2}, \Pi^{\nu_1},\tilde{g}^{f_{1}})$, where  $\lbrace dx_{1}, . . . dx_{k_{1}}\rbrace$ are $\tilde{g}_{1}$-  orthonormal basis on $M_{1}$ and  $\lbrace dy_{1}, . . .   dy_{k_{2}} \rbrace$ are  $\tilde{g}_{2}$-  orthonormal basis on $M_{2}$. If $\left( M_{1}, \bar{g}_{1}\right) $ have index $q_{1}$ with local $\bar{g}_{1}$ - orthonormal basis $\lbrace \frac{\partial}{\partial x_{1}}, . . . \frac{\partial}{\partial x_{k_{1}}} \rbrace$, such that $\bar{g}_{1}\left( \frac{\partial}{\partial x_{i}}, \frac{\partial}{\partial x_{i}} \right) = - 1 $, for all $i \in \lbrace1,. . .,q_{1} \rbrace$ and $\bar{g}_{1}\left( \frac{\partial}{\partial x_{i}}, \frac{\partial}{\partial x_{i}} \right) =  1 $, for all $i \in \lbrace q_{1} + 1 ,. . .,k_{1} \rbrace$, then from equations (\ref{8.2.1})  and (\ref{8.2.3}), we have
\begin{equation}\label{5.3.12}
	\begin{cases}
		\tilde{g}^{f_{1}}(dx_i^h,dx_i^h)=\tilde{g}_1(dx_i,dx_i)^h = \bar{g}_{1}\left(\sharp_{\bar{g}_{1}} (dx_i),  \sharp_{\bar{g}_{1}} (dx_i) \right)^h\\
	\hspace*{3.2cm}	= \bar{g}_{1}\left(\frac{\partial}{\partial x^i},  \frac{\partial}{\partial x^i} \right)^h = - 1, \qquad \forall \; i \in \lbrace1,. . .,q_{1} \rbrace,\\
 \tilde{g}^{f_{1}}(dx_i^h,dx_i^h) = \bar{g}_{1}\left(\frac{\partial}{\partial x^i},  \frac{\partial}{\partial x^i} \right)^h =  1, \qquad \forall\;  i \in \lbrace q_{1} + 1 ,. . .,k_{1} \rbrace.
\end{cases}
\end{equation}

 Similarly,  if $\left( M_{2}, \bar{g}_{2}\right) $ be a pseudo-Riemannian manifold of index $q_{2}$ with local $\bar{g}_{2}$ - orthonormal basis $\lbrace \frac{\partial}{\partial y_{1}}, . . . \frac{\partial}{\partial y_{k_{2}}} \rbrace$, $\bar{g}_{2}\left( \frac{\partial}{\partial y_{i}}, \frac{\partial}{\partial y_{i}} \right) = - 1 $, for all $i \in \lbrace1,. . .,q_{2} \rbrace$ and $\bar{g}_{2}\left( \frac{\partial}{\partial y_{i}}, \frac{\partial}{\partial y_{i}} \right) =  1 $, for all $i \in \lbrace q_{2} + 1 ,. . .,k_{2} \rbrace$, then from  (\ref{8.2.1})  and (\ref{8.2.3}), we have

\begin{equation}\label{5.3.13}
	\begin{cases}
		\tilde{g}^{f_{1}}({f_{1}}^h dy_i^v,{f_{1}}^h dx_i^v)=\tilde{g}_2(dy_i,dy_i)^v = \bar{g}_{2}\left(\sharp_{\bar{g}_{2}} (dy_i),  \sharp_{\bar{g}_{2}} (dy_i) \right)^h\\
	 \hspace*{3.2cm}	= \bar{g}_{2}\left(\frac{\partial}{\partial y^i},  \frac{\partial}{\partial y^i} \right)^h = - 1, \qquad \forall \; i \in \lbrace1,. . .,q_{2} \rbrace,\\
		\tilde{g}^{f_{1}}({f_{1}}^h dx_i^v,{f_{1}}^hdx_i^v ) = \bar{g}_{1}\left(\frac{\partial}{\partial y^i},  \frac{\partial}{\partial y^i} \right)^v =  1, \qquad \forall \; i \in \lbrace q_{2} + 1 ,. . .,k_{2} \rbrace.
	\end{cases}
\end{equation}

Thus from equations (\ref{5.3.12}) and (\ref{5.3.13}), if $\left( M_{1}, \Pi_{1}\right) $ and $\left( M_{2}, \Pi_{2}\right) $ are semi-Riemannian Poisson manifolds of index $q_{1}$ and $q_{2}$, respectively, then contravariant warped product space $(M= M_{1}\times_{f_{1}} M_{2}, \Pi^{\nu_1},\tilde{g}^{f_{1}})$ will have index $q_{1} + q_{2}$. Now, for local $\tilde{g}^{f_{1}}$-  orthonormal basis $\lbrace dx^h_{1}, . . . dx^h_{k_{1}}, {f_{1}}^h dy_{1}^v, . . .  {f_{1}}^h dy_{k_{2}}^v \rbrace$, 
\begin{equation}\label{5.3.14}
	\begin{cases}
		\tilde{g}^{f_{1}}({f_{1}}^h dy_i^v,{f_{1}}^h dx_i^v)=  - 1 = \tilde{g}^{f_{1}}(dx_j^h,dx_j^h), \qquad \forall \; i \in \lbrace1,. . .,q_{2} \rbrace, j \in \lbrace1,. . .,q_{1} \rbrace ,\\
		\tilde{g}^{f_{1}}({f_{1}}^h dx_i^v,{f_{1}}^hdx_i^v ) =  1 = \tilde{g}^{f_{1}}(dx_j^h,dx_j^h),\qquad \forall \; i \in \lbrace q_{2} + 1 ,. . .,k_{2} \rbrace, j \in \lbrace q_{1} + 1,. . .,k_{1} \rbrace.
	\end{cases}
\end{equation} 

 The contravariant  Laplacian ($\Delta ^ \mathcal{D}$) for any tensor $T$ and Hessian ($H^{u}_{\Pi}$) of a smooth function $u$ on contravariant warped product space $(M= M_{1}\times_{f_{1}} M_{2}, \Pi^{\nu_1},\tilde{g}^{f_{1}})$, are given by 
 \begin{equation}\label{8.3.15}
 	\begin{cases}
 		\Delta ^ \mathcal{D} (T) = - \sum_{i = 1}^{k_{1}} \mathcal{D}^2_{dx^h_{i}, dx^h_{i}} T - \sum_{j = 1}^{k_{2}} \mathcal{D}^2_{{f_{1}}^h dy^v_{j}, {f_{1}}^h dy^v_{j}} T,\\
 		H^{u}_{\Pi}(\theta_{i}, \theta_{i}) = \tilde{g}\left(\mathcal{D}_{\theta_{i}} J_{1}du, \theta_{i} \right), \; \; \theta_{i} \in \lbrace dx^h_{1}, . . . dx^h_{k_{1}}, {f_{1}}^h dy_{1}^v, . . .  {f_{1}}^h dy_{k_{2}}^v \rbrace. 
 		\end{cases}
 \end{equation}
 Also, the second order contravariant derivative of $(r,s)-type$ type tensor field $T$ on $M$ is defined as 
 \begin{eqnarray}
 	(\mathcal{D}_{\omega,\eta}^2T)(\gamma_1,...,\gamma_r,A_1,...,A_s) &=(\mathcal{D}_\omega({\mathcal{D}}P))(\eta,\gamma_1,...,\gamma_r,A_1,...,A_s) \nonumber\\
 	&=(\mathcal{D}_\omega(\mathcal{D}_\eta P))(\gamma_1,...,\gamma_r,A_1,...,A_s) \nonumber\\
 	&-(\mathcal{D}_{\mathcal{D}_\omega \eta}P)(\gamma_1,...,\gamma_r,A_1,...,A _s),	
 \end{eqnarray}
 where $\omega,\eta, \gamma_1,...,\gamma_r\in\Omega^1(M)$ and $A_1,...,A_s\in\mathfrak{X}(M)$. 
\begin{theorem}
Let $(M= M_{1}\times_{f_{1}} M_{2}, \Pi^{\nu_1},\tilde{g}^{f_{1}})$ be a pseudo-Riemannian contravariant warped product space and $u \in C^{\infty} M$. Then 
$$\Delta ^{\mathcal{D}}(u) = \left( \Delta ^ {\mathcal{D}^{{M_{1}}}} u_{1}\right)^h + \frac{\kappa_{2} - 2 q_{2}}{{f_{1}}^h}\tilde{g}_{1}\left( J_{1} d{f_{1}}, du_{1}\right)^h  + \left(\nu_{1}^h{f_{1}}^h \right)^2 \left( \Delta ^ {\mathcal{D}^{M_{2}}} u_{2}\right)^v,$$

where $u_{1}^h$ ($u_{2}^v$)  horizontal (vertical) lift of $u_{1} \in C^{\infty} M_{1}$ $\left( u_{2} \in C^{\infty} M_{2}\right) $ to $M$, and $u = u_{1}^h + u_{2}^v$.
\end{theorem} 
\begin{proof}
	Let $u$ be a smooth function on pseudo-Riemannian contravariant warped product space $(M= M_{1}\times_{f_{1}} M_{2}, \Pi^{\nu_1},\tilde{g}^{f_{1}})$, and 
	$u = u_{1}^h + u_{2}^v$, where $u_{1}^h$ ($u_{2}^v$)  horizontal (vertical) lift of $u_{1} \in C^{\infty} M_{1}$ $\left( u_{2} \in C^{\infty} M_{2}\right) $ to $C^{\infty} M$. Then
	\begin{equation}\label{8.3.17}
			\Delta ^ \mathcal{D}u = 	\Delta ^ \mathcal{D} u_{1}^h + 	\Delta ^ \mathcal{D} u_{2}^v.
	\end{equation} 
Now, using equation (\ref{8.3.15}), Proposition (\ref{p2.3}) and Proposition (\ref{p2.4}), we obtain
\begin{align}\label{8.3.18}
\nonumber	\Delta ^ \mathcal{D} u_{1}^h &=  - \sum_{i = 1}^{k_{1}} \mathcal{D}^2_{dx^h_{i}, dx^h_{i}} u_{1}^h - \sum_{j = 1}^{k_{2}} \mathcal{D}^2_{{f_{1}}^h dy^v_{j}, {f_{1}}^h dy^v_{j}} u_{1}^h\\
\nonumber &= - \sum_{i = 1}^{k_{1}} \mathcal{D}_{dx^h_{i}}\left( \mathcal{D}_{dx^h_{i}}u_{1}^h\right)  + \sum_{i = 1}^{k_{1}} \mathcal{D}_{\mathcal{D}_{dx^h_{i}} dx^h_{i} } u_{1}^h\\
\nonumber & \qquad- \sum_{j = 1}^{k_{2}} \mathcal{D}_{{f_{1}}^h dy^v_{j}}\left( \mathcal{D}_{{f_{1}}^h dy^v_{j}}u_{1}^h\right)  + \sum_{j = 1}^{k_{2}} \mathcal{D}_{\mathcal{D}_{{f_{1}}^h dy^v_{j}} \left( {f_{1}}^h dy^v_{j}\right)  } u_{1}^h\\
\nonumber &=  - \sum_{i = 1}^{k_{1}} \left( \mathcal{D}^{M_{1}}_{dx_{i}}\left( \mathcal{D}^{M_{1}}_{dx_{i}}u_{1}\right)  - \sum_{i = 1}^{k_{1}} \mathcal{D}^{M_{1}}_{\mathcal{D}^{M_{1}}_{dx_{i}} dx_{i} } u_{1}\right)^h  + \left( {f_{1}}^h\right)^2 \sum_{j = 1}^{k_{2}} \mathcal{D}_{\mathcal{D}_{ dy^v_{j}} \left(  dy^v_{j}\right)  } u_{1}^h\\
\nonumber &= \left( \Delta ^ {\mathcal{D}^{{M_{1}}}} u_{1}\right)^h + \frac{1}{{f_{1}}^h} \sum_{j = 1}^{k_{2}}g_{2} \left(dy_{j}, dy_{j} \right)^v g_{1}\left( J_{1}d{f_{1}}, J_{1}du_{1}\right)^h\\ 
&= \left( \Delta ^ {\mathcal{D}^{{M_{1}}}} u_{1}\right)^h + \frac{\kappa_{2} - 2 q_{2}}{{f_{1}}^h}\tilde{g}_{1}\left( J_{1} d{f_{1}}, du_{1}\right)^h,
\end{align}
 and similarly,
\begin{align}\label{8.3.19}
	\nonumber	\Delta ^ \mathcal{D} u_{2}^v &=  - \sum_{i = 1}^{k_{1}} \mathcal{D}^2_{dx^h_{i}, dx^h_{i}} u_{2}^v - \sum_{j = 1}^{k_{2}} \mathcal{D}^2_{{f_{1}}^h dy^v_{j}, {f_{1}}^h dy^v_{j}} u_{2}^v\\
	\nonumber &= \left(\nu_{1}^h{f_{1}}^h \right)^2 \left( - \sum_{j = 1}^{k_{2}} \mathcal{D}^{M_{2}}_{ dy^v_{j}}\left( \mathcal{D}^{M_{2}}_{ dy^v_{j}}u_{2}^v\right)  + \sum_{j = 1}^{k_{2}} \mathcal{D}^{M_{2}}_{\mathcal{D}^{M_{2}}_{dy^v_{j}} \left(  dy^v_{j}\right)  } u_{2}^v\right) \\
 &=  \left(\nu_{1}^h{f_{1}}^h \right)^2 \left( \Delta ^ {\mathcal{D}^{{M_{2}}}} u_{2}\right)^v.
\end{align} 
Thus equations (\ref{8.3.17}), (\ref{8.3.18}) and (\ref{8.3.19}) completes the proof.
\end{proof}    
\section{\textbf{Qualar curvatures and null sectional curvatures}}
Let $\left(M_{1}, \bar{g}_{1} \right)$ be a pseudo-Riemannian manifold of index $q_{1}$, and local $\bar{g}_{1}$ - orthonormal basis  $\lbrace \frac{\partial}{\partial x_{1}}, . . . \frac{\partial}{\partial x_{k_{1}}} \rbrace$, such that $\bar{g}_{1}\left( \frac{\partial}{\partial x_{i}}, \frac{\partial}{\partial x_{i}} \right) = - 1 $, for all $i \in \lbrace1,. . .,q_{1} \rbrace$ and $\bar{g}_{1}\left( \frac{\partial}{\partial x_{s}}, \frac{\partial}{\partial x_{s}} \right) =  1 $, for all $s \in \lbrace q_{1} + 1 ,. . .,k_{1} \rbrace$. Then the qualar curvature $(qual)$ of $\left(M_{1}, \bar{g}_{1} \right)$ is defined as a sum of scalar curvatures of some plane section at a point $p$, of $\left(M_{1}, \bar{g}_{1} \right)$, such that \cite{8mgl},
\begin{equation}\label{8.4.1}
	qual^1(p) = 2\sum_{i = 1}^{q_{1}}\sum_{s = q_{1} + 1}^{k_{1}} K^1(\frac{\partial}{\partial x_{i}}, \frac{\partial}{\partial x_{s}}).
\end{equation}

Whereas, the null sectional curvature $(\bar{K}^1 )$ of a degenerate plane spanned by null vector field $N^i_{s} = \frac{1}{\sqrt{2}} \left( \frac{\partial}{\partial x_{i}}, \frac{\partial}{\partial x_{s}}\right)$ and non null vector field $\frac{\partial}{\partial x_{l}}$, $l \neq i$, $l \neq s$, is given by  
\cite{8mgl, 8ahb},
	\begin{align}\label{8.4.2}
	\nonumber	\bar{K}^1 &=\frac{\bar{g}_{1}\left( R\left(N^i_{s}, \frac{\partial}{\partial x_{l}} \right)\frac{\partial}{\partial x_{l}}, N^i_{s}\right) }{\bar{g}_{1} \left(\frac{\partial}{\partial x_{l}}, \frac{\partial}{\partial x_{l}}\right)},\\
	 &= - \frac{1}{2}\bar{K}^1\left(\frac{\partial}{\partial x_{i}}, \frac{\partial}{\partial x_{l}} \right)  + \frac{1}{2}\bar{K}^1\left(\frac{\partial}{\partial x_{s}}, \frac{\partial}{\partial x_{l}} \right) + \epsilon_{l} \bar{g}_{1}\left( R\left(\frac{\partial}{\partial x_{i}}, \frac{\partial}{\partial x_{l}} \right)\frac{\partial}{\partial x_{l}}, \frac{\partial}{\partial x_{s}} \right).
	\end{align}
	
where $\bar{g}_{1}\left( \frac{\partial}{\partial x_{l}}, \frac{\partial}{\partial x_{l}}\right) = \epsilon_{l}$. Thus, we can 	generalize the concepts of qualar curvatures and null sectional curvatures from $\left(M_{1}, \bar{g}_{1} \right)$ to $\left(M_{1}, \tilde{g}_{1}, \Pi_{1} \right)$, where  $\tilde{g}_{1}$ is a cometric of $\bar{g}_{1}$. Therefore from  equations $(\ref{5.3.13})$ and $(\ref{8.4.1})$, we define  the qualar curvature $\left(\mathcal{Q}^1 \right) $ at a point   $p \in \left(M_{1}, \tilde{g}_{1}, \Pi_{1} \right)$, as follows
\begin{equation}\label{8.4.3}
	\mathcal{Q}^1(p) =  2\sum_{i = 1}^{q_{1}}\sum_{s = q_{1} + 1}^{k_{1}} \mathcal{K}^1(dx_{i}, dx_{s}).
\end{equation}

Also for smooth sections $\xi^i_{s} = \frac{1}{\sqrt{2}} \left( dx_{i} + dx_{s}\right), \bar{\xi^i_{s}} = \frac{1}{\sqrt{2}} \left( - dx_{i} + dx_{s}\right)$, $i \in \lbrace 1, . . . q_{1} \rbrace $, $s \in \lbrace q_{1} + 1, . . . k_{1} \rbrace $ in $T^*M_{1}$, we can see that
\begin{equation}\label{8.4.4}
	\bar{g}_{1}\left( \xi^i_{s}, \xi^i_{s}\right) = 0 = \bar{g}_{1}\left( \bar{\xi}^i_{s}, \bar{\xi}^i_{s}\right), \qquad and\qquad \bar{g}_{1}\left( \xi^i_{s}, \bar{\xi}^i_{s}\right) = 1 .
\end{equation}

Thus we define,   null sectional curvature $\left( \mathcal{N}^1\right) $ of  the plane spanned by null one form $\xi^i_{s}$ and a one form $dx_{l}$, $l \neq i$, $l \neq s$, on $M$, in the following way
\begin{align}\label{8.4.5}
		\mathcal{K}^1\left( \xi^i_{s}, dx_{l}\right)  &=\frac{\tilde{g}_{1}\left( \mathcal{R}\left(\xi^i_{s}, dx_{l} \right)dx_{l}, \xi^i_{s}\right) }{\tilde{g}_{1} \left(dx_{l}, dx_{l}\right)}.
\end{align}

Let $(M= M_{1}\times_{f_{1}} M_{2}, \Pi^{\nu_1},\tilde{g}^{f_{1}})$, be a contravariant  warped product space   with local $\tilde{g}^{f_1}$-  orthonormal basis  $\lbrace dx^h_{1}, . . . dx^h_{k_{1}}, f_{1}^h dy_{1}^v, . . .  f_{k_2}^h dy_{k_{2}}^v \rbrace$, where  $\lbrace dx_{1}, . . . dx_{k_{1}}\rbrace$ is $\tilde{g}_{1}$-  orthonormal basis on $M_{1}$ of index $q_{1}$ and  $\lbrace dy_{1}, . . .   dy_{k_{2}} \rbrace$ are  $\tilde{g}_{2}$-  orthonormal basis on $M_{2}$ of index $q_{2}$. Then qualar curvatures $\mathcal{Q}$ at $p = \left( p_{1}, p_{2}\right)  \in M$,  and null sectional curvatures $\left( \mathcal{K}\right)$ on $M$ are respectively defined by the equations
\begin{align}\label{8.4.6}
\nonumber	\mathcal{Q}(p) = & 2\sum_{i = 1}^{q_{1}}\sum_{s = q_{1} + 1}^{k_{1}} \mathcal{K}^M(dx_{i}^h, dx_{s}^h) + 2\sum_{i = 1}^{q_{1}}\sum_{s = q_{2} + 1}^{k_{2}} \mathcal{K}^M(dx_{i}, {f_{1}}^hdy_{s})\\
	& + 2\sum_{i = 1}^{q_{2}}\sum_{s = q_{1} + 1}^{k_{1}} \mathcal{K}^M({f_{1}}^hdy_{i}, dx_{s}) + 2\sum_{i = 1}^{q_{2}}\sum_{s = q_{2} + 1}^{k_{2}} \mathcal{K}^M({f_{1}}^hdy_{i}, {f_{1}}^hdy_{s}),
\end{align}
and
\begin{align}\label{8.4.7}
	\mathcal{K}\left( \xi^i_{s}, d\theta_{l}\right)  &=\frac{\tilde{g}^{f_1}\left( \mathcal{R}\left(\xi^i_{s}, d\theta_{l} \right)d\theta_{l}, \xi^i_{s}\right) }{\tilde{g}^{f_1} \left(d\theta_{l}, d\theta_{l}\right)},
\end{align}
where $\xi^i_{s} \in \lbrace\xi^{i,h}_{s,h}, \xi^{i,h}_{s,v},\xi^{i,v}_{s,v}\rbrace$,
$d\theta_{l} \in \lbrace dx^h_{1}, . . . dx^h_{k_{1}}, f^h dy_{1}^v, . . .  f^h dy_{k_{2}}^v \rbrace$, $i \neq l$, $i \neq s$ and
\begin{equation}\label{8.4.8}
	\begin{cases}
		\xi^{i,h}_{s,h} =  \frac{1}{\sqrt{2}} \left( dx_{i}^h + dx_{s}^h\right), \qquad \forall\; i \in \lbrace 1, . . . q_{1} \rbrace , s \in \lbrace q_{1} + 1, . . . k_{1} \rbrace,\\
		\xi^{i,h}_{s,v} =  \frac{1}{\sqrt{2}} \left( dx_{i}^h + {f_{1}}^hdy_{s}^v\right) \;\qquad \forall i \in \lbrace 1, . . . q_{1} \rbrace , s \in \lbrace q_{2} + 1, . . . k_{2} \rbrace,\\
		\xi^{i,v}_{s,h} =  \frac{1}{\sqrt{2}} \left({f_{1}}^hdy_{i}^v + dx_{s}^h \right) \;\qquad \forall i \in \lbrace 1, . . . q_{2} \rbrace , s \in \lbrace q_{1} + 1, . . . k_{1} \rbrace,\\
		\xi^{i,v}_{s,v} =  \frac{1}{\sqrt{2}} \left( {f_{1}}^hdy_{i}^v + {f_{1}}^hdy_{s}^v\right) \qquad \forall i \in \lbrace 1, . . . q_{2} \rbrace , s \in \lbrace q_{2} + 1, . . . k_{2} \rbrace,
	\end{cases}
\end{equation}
such that $\tilde{g}^{f_{1}}  \left(\xi^{i,h}_{s,h}, \xi^{i,h}_{s,h} \right)  =  \tilde{g}^{f_{1}}  \left(\xi^{i,h}_{s,v}, \xi^{i,h}_{s,v} \right) = \tilde{g}^{f_{1}}  \left(\xi^{i,v}_{s,v}, \xi^{i,v}_{s,v} \right) = 0$.

  \begin{proposition}
   Let  $(M= M_{1}\times_{f_{1}} M_{2}, \Pi^{\nu_1},\tilde{g}^{f_{1}})$, be a contravariant  warped product space. Then null sectional curvature of a degenerate plane on $M$ holds any one of the following relations 
 \begin{equation*}
(1).\; \;\mathcal{K}\left( \xi^{i,h}_{s,h}, \eta_{l}\right) =
	\begin{cases}
		- \frac{1}{2} \mathcal{K}^{M}\left(dx_{i}^h, dx_{l}^h  \right)  + \frac{1}{2} \mathcal{K}^{M}\left(dx_{s}^h, dx_{l}^h  \right) \\ 
		
		\hspace{2.4cm}+ \epsilon_{l_{1}} \tilde{g}^{f_{1}}\left(\mathcal{R}\left( dx_{i}^h, dx_{l}^h\right)dx_{l}^h, dx_{s}^h \right),   
		& if\; \eta_{l}  = dx_{l}^h, \\
		
	- \frac{1}{2} \mathcal{K}^{M}\left(dx_{i}^h, dy_{l}^v  \right)  + \frac{1}{2} \mathcal{K}^{M}\left(dx_{s}^h, dy_{l}^v  \right) \\ 
	
	\hspace{2.4cm}+ \epsilon_{l_{2}}(f_{1}^h)^2 \tilde{g}^{f_{1}} \left(\mathcal{R}\left( dx_{i}^h, dy_{l}^v\right)dy_{l}^v, dx_{s}^h \right),   
	& if\; \eta_{l}  = {f_{1}}^hdy_{l}^v, \\
	\end{cases}
\end{equation*}
where $\tilde{g}_{1}\left(dx_{l}, dx_{l} \right) = \epsilon_{l_{1}} = \pm 1$, $\tilde{g}_{2}\left(dy_{l}, dy_{l} \right) = \epsilon_{l_{2}} = \pm 1$,  $i \in \lbrace 1, . . . q_{1} \rbrace $, $s \in \lbrace q_{1} + 1, . . . k_{1} \rbrace $, $l \neq i$ and $l \neq s$.
 \begin{equation*}
	(2).\; \;\mathcal{K}\left( \xi^{i,h}_{s,v}, \eta_{l}\right) =
	\begin{cases}
		- \frac{1}{2} \mathcal{K}^{M}\left(dx_{i}^h, dx_{l}^h  \right)  + \frac{1}{2} \mathcal{K}^{M}\left(dy_{s}^v, dx_{l}^h  \right) \\ 
		\hspace{2.4cm}+ \epsilon_{l_{1}} (f_{1}^h) \tilde{g}^{f_{1}}\left(\mathcal{R}\left( dx_{i}^h, dx_{l}^h\right)dx_{l}^h, dy_{s}^v \right),   
		& if\; \eta_{l}  = dx_{l}^h, \\
		- \frac{1}{2} \mathcal{K}^{M}\left(dx_{i}^h, dy_{l}^v  \right)  + \frac{1}{2} \mathcal{K}^{M}\left(dy_{s}^v, dy_{l}^v  \right) \\ 
		\hspace{2.4cm}+ \epsilon_{l_{2}}(f_{1}^h)^3 \tilde{g}^{f_{1}}\left(\mathcal{R}\left( dx_{i}^h, dy_{l}^v\right)dy_{l}^v, dy_{s}^v \right),   
		& if\; \eta_{l}  = {f_{1}}^hdy_{l}^v, \\
	\end{cases}
\end{equation*}
where $\tilde{g}_{1}\left(dx_{l}, dx_{l} \right) = \epsilon_{l_{1}} = \pm 1$, $\tilde{g}_{2}\left(dy_{l}, dy_{l} \right) = \epsilon_{l_{2}} = \pm 1$,  $i \in \lbrace 1, . . . q_{1} \rbrace $, $s \in \lbrace q_{2} + 1, . . . k_{2} \rbrace $, $l \neq i$ and $l \neq s$.
\begin{equation*}
	(3).\; \;\mathcal{K}\left( \xi^{i,v}_{s,h}, \eta_{l}\right) =
	\begin{cases}
		 \frac{1}{2} \mathcal{K}^{M}\left(dx_{i}^h, dx_{l}^h  \right)  - \frac{1}{2} \mathcal{K}^{M}\left(dy_{s}^v, dx_{l}^h  \right) \\ 
		\hspace{2.4cm}+ \epsilon_{l_{1}} ({f_{1}}^h) \tilde{g}^{f_{1}}\left(\mathcal{R}\left( dx_{i}^h, dx_{l}^h\right)dx_{l}^h, dy_{s}^v \right),   
		& if\; \eta_{l}  = dx_{l}^h, \\
		\frac{1}{2} \mathcal{K}^{M}\left(dx_{i}^h, dy_{l}^v  \right)  - \frac{1}{2} \mathcal{K}^{M}\left(dy_{s}^v, dy_{l}^v  \right) \\ 
		\hspace{2.4cm}+ \epsilon_{l_{2}}({f_{1}}^h)^3 \tilde{g}^{f_{1}}\left(\mathcal{R}\left( dx_{i}^h, dy_{l}^v\right)dy_{l}^v, dy_{s}^v \right),   
		& if\; \eta_{l}  = {f_{1}}^hdy_{l}^v, \\
	\end{cases}
\end{equation*}
where $\tilde{g}_{1}\left(dx_{l}, dx_{l} \right) = \epsilon_{l_{1}} = \pm 1$, $\tilde{g}_{2}\left(dy_{l}, dy_{l} \right) = \epsilon_{l_{2}} = \pm 1$,  $i \in \lbrace 1, . . . q_{2} \rbrace $, $s \in \lbrace q_{1} + 1, . . . k_{1} \rbrace $, $l \neq i$ and $l \neq s$.
 \begin{equation*}
	(4).\; \;\mathcal{K }\left( \xi^{i,v}_{s,v}, \eta_{l}\right) =
	\begin{cases}
		- \frac{1}{2} \mathcal{K}^{M}\left(dy_{i}^v, dx_{l}^h  \right)  + \frac{1}{2} \mathcal{K}^{M}\left(dy_{s}^v, dx_{l}^h  \right) \\ 
		\hspace{2.4cm}+ \epsilon_{l_{1}}({f_{1}}^h)^2 \tilde{g}^{f_{1}}\left(\mathcal{R}\left( dy_{i}^v, dx_{l}^h\right)dx_{l}^h, dy_{s}^v \right),   
		& if\; \eta_{l}  = dx_{l}^h, \\
		- \frac{1}{2} \mathcal{K}^{M}\left(dy_{i}^v, dy_{l}^v  \right)  + \frac{1}{2} \mathcal{K}^{M}\left(dy_{s}^v, dy_{l}^v  \right) \\ 
		\hspace{2.4cm}+ \epsilon_{l_{2}}({f_{1}}^h)^4 \tilde{g}^{f_{1}}\left(\mathcal{R}\left( dy_{i}^v, dy_{l}^v\right)dy_{l}^v, dy_{s}^v \right),   
		& if\; \eta_{l}  = {f_{1}}^hdy_{l}^v, \\
	\end{cases}
\end{equation*}
where $\tilde{g}_{1}\left(dx_{l}, dx_{l} \right) = \epsilon_{l_{1}} = \pm 1$, $\tilde{g}_{2}\left(dy_{l}, dy_{l} \right) = \epsilon_{l_{2}} = \pm 1$  $i \in \lbrace 1, . . . q_{2} \rbrace $, $s \in \lbrace q_{2} + 1, . . . k_{2} \rbrace $, $l \neq i$ and $l \neq s$.
\end{proposition}   
\begin{proof}
	From $(\ref{8.4.7})$, for all $i \in \lbrace 1, . . . q_{1} \rbrace $, $s \in \lbrace q_{1} + 1, . . . k_{1} \rbrace $,  $l \neq i$ and $l \neq s$, we get the relations
	 \begin{align*}
	 	\mathcal{K}\left( \xi^{i,h}_{s,h}, dx^h_{l}\right)  &=\frac{\tilde{g}^{f_1}\left( \mathcal{R}\left(\xi^i_{s}, dx^h_{l} \right)dx^h_{l}, \xi^{i,h}_{s,h} \right) }{\tilde{g}^{{f_{1}}} \left(dx^h_{l}, dx^h_{l}\right)}
	 	=\frac{1}{2\tilde{g}^{f_{1}} \left(dx^h_{l}, dx^h_{l}\right)}\lbrace \tilde{g}^{f_{1}} \left(\mathcal{R}\left( dx_{i}^h, dx_{l}^h\right)dx_{l}^h, dx_{s}^v \right)\\
	 	 & \qquad + 2 \tilde{g}^{f_{1}} \left(\mathcal{R}\left( dx_{i}^h, dx_{l}^h\right)dx_{l}^h, dx_{s}^h \right) +  \tilde{g}^{f_{1}} \left(\mathcal{R}\left( dx_{i}^h, dx_{l}^h\right)dx_{l}^h, dx_{s}^h \right)  \rbrace \\
	 	 &= - \frac{1}{2} \mathcal{K}^{M}\left(dx_{i}^h, dx_{l}^h  \right)  + \frac{1}{2} \mathcal{K}^{M}\left(dx_{s}^h, dx_{l}^h  \right)  
	 	 + \epsilon_{l_{1}}   \tilde{g}^{f_{1}}  \left(\mathcal{R}\left( dx_{i}^h, dx_{l}^h\right)dx_{l}^h, dx_{s}^h \right), 
	 \end{align*} 
 and 
  \begin{align*}
 	\mathcal{K}\left( \xi^{i,h}_{s,h}, {f_{1}}^h dy^v_{l}\right)  &=\frac{\tilde{g}^{f_1}\left( \mathcal{R}\left(\xi^i_{s}, {f_{1}}^hdy^v_{l} \right){f_{1}}^h dy^v_{l}, \xi^{i,h}_{s,h} \right) }{\tilde{g}^{f_{1}} \left({f_{1}}^hdy^v_{l}, dy^v_{l}\right)}
 	=\frac{1}{2\tilde{g}^{f_{1}} \left(dy^v_{l}, dy^v_{l}\right)}\lbrace \tilde{g}^{f_{1}}  \left(\mathcal{R}\left( dx_{i}^h, dy_{l}^v\right)dy_{l}^v, dx_{s}^v \right)\\
 	& \qquad + 2 \tilde{g}^{f_{1}}  \left(\mathcal{R}\left( dx_{i}^h, dy_{l}^v\right)dy_{l}^h, dx_{s}^h \right) +  \tilde{g}^{f_{1}}  \left(\mathcal{R}\left( dx_{i}^h, dy_{l}^v\right)dy_{l}^v, dx_{s}^h \right)  \rbrace \\
 	&= - \frac{1}{2} \mathcal{K}^{M}\left(dx_{i}^h, dy_{l}^v  \right)  + \frac{1}{2} \mathcal{K}^{M}\left(dx_{s}^h, dy_{l}^v  \right)  
 	+ \epsilon_{l_{2}} ({f_{1}}^h)^2 \tilde{g}^{f_{1}} \left(\mathcal{R}\left( dx_{i}^h, dx_{l}^h\right)dx_{l}^h, dx_{s}^h \right). 
 \end{align*} 
which together prove the first part of the proposition. Similarly, we can prove the remaining parts of the proposition.
\end{proof} 
If $\xi^{i,h}_{s,h} =  \frac{1}{\sqrt{2}} \left( dx_{i}^h + dx_{s}^h\right)$, then we can consider $\bar{\xi}^{i,h}_{s,h} =  \frac{1}{\sqrt{2}} \left( - dx_{i} + dx_{s}\right),$ such that $\tilde{g}^{f_{1}} \left( \bar{\xi}^{i,h}_{s,h}, \bar{\xi}^{i,h}_{s,h} \right) = 0 = \tilde{g}^{f_{1}} \left( \xi^{i,h}_{s,h}, \xi^{i,h}_{s,h} \right) $ and $ \tilde{g}^{f_{1}} \left( \xi^{i,h}_{s,h}, \bar{\xi}^{i,h}_{s,h}\right) = 1  $ for all  $i \in \lbrace 1, . . . q_{1} \rbrace , s \in \lbrace q_{1} + 1, . . . k_{1} \rbrace $. Similarly for remaining $\xi^{i}_{s}$, we can find their counter fields $\bar{\xi}^{i}_{s}$ such that  $ g^{f_{1}} \left( \xi^{i}_{s}, \xi^{i}_{s} \right) = 0 = \tilde{g}^{f_{1}} \left( \bar{\xi}^{i}_{s}, \bar{\xi}^{i}_{s} \right) $ and $\tilde{g}^{f_{1}} \left( \xi^{i}_{s}, \bar{\xi}^{i}_{s} \right) = 1$. Now, using equalities
\begin{equation}\label{8.4.9}
	\begin{cases}
			\mathcal{K}^{M}(dx_{i}^h, dx_{s}^h) = \mathcal{K}^{M}(\xi^{i,h}_{s,h}, \bar{\xi}^{i,h}_{s,h}),\qquad \forall\; i \in \lbrace 1, . . . q_{1} \rbrace , s \in \lbrace q_{1} + 1, . . . k_{1} \rbrace,\\
			\mathcal{K}^{M}(dx_{i}^h, f_{1}^hdy_{s}^v) = \mathcal{K}^{M}(\xi^{i,h}_{s,v}, \bar{\xi}^{i,h}_{s,v}),\qquad \forall\; i \in \lbrace 1, . . . q_{1} \rbrace , s \in \lbrace q_{2} + 1, . . . k_{2} \rbrace,\\
			\mathcal{K}^{M}(dx_{s}^h, f_{1}^hdy_{i}^v) = \mathcal{K}^{M}(\xi^{i,v}_{s,h}, \bar{\xi}^{i,v}_{s,h}),\qquad \forall\; i \in \lbrace 1, . . . q_{2} \rbrace , s \in \lbrace q_{1} + 1, . . . k_{1} \rbrace,\\
			\mathcal{K}^{M}(f_{1}^hdy_{i}^v, f_{1}^hdy_{s}^v) = \mathcal{K}^{M}(\xi^{i,v}_{s,v}, \bar{\xi}^{i,v}_{s,v}), \qquad \forall i \in \lbrace 1, . . . q_{2} \rbrace , s \in \lbrace q_{2} + 1, . . . k_{2} \rbrace,
	\end{cases}
\end{equation}
  in equation (\ref{8.4.6}), we obtain the following proposition.
\begin{proposition}
	Let  $(M= M_{1}\times_{f_{1}} M_{2}, \Pi^{\nu_1},\tilde{g}^{f_{1}})$, be a contravariant  warped product space. Then  qualar  curvature is  the sum of some sectional curvatures of the planes spanned by null $1 -$ forms $\xi^{i}_{s}, \bar{\xi}^{i}_{s} \in \Gamma(T^*M)$ such that $g^{f_{1}} \left( \xi^{i}_{s}, \bar{\xi}^{i}_{s} \right) = 1$, and holds the relation 
	\begin{align}\label{8.4.10}
		\nonumber	\mathcal{Q}(p) = & 2\sum_{i = 1}^{q_{1}}\sum_{s = q_{1} + 1}^{k_{1}} \mathcal{K}^{M}(\xi^{i,h}_{s,h}, \bar{\xi}^{i,h}_{s,h}) + 2\sum_{i = 1}^{q_{1}}\sum_{s = q_{2} + 1}^{k_{2}}\mathcal{K}^{M}(\xi^{i,h}_{s,v}, \bar{\xi}^{i,h}_{s,v})\\
		& + 2\sum_{i = 1}^{q_{2}}\sum_{s = q_{1} + 1}^{k_{1}}\mathcal{K}^{M}(\xi^{i,v}_{s,h}, \bar{\xi}^{i,v}_{s,h}) + 2\sum_{i = 1}^{q_{2}}\sum_{s = q_{2} + 1}^{k_{2}} \mathcal{K}^{M}(\xi^{i,v}_{s,v}, \bar{\xi}^{i,v}_{s,v}).
	\end{align}  
\end{proposition}    
\begin{theorem}\label{x123}
		Let  $(M= M_{1}\times_{f_{1}} M_{2}, \Pi^{\nu_1},\tilde{g}^{f_{1}})$, be a contravariant  warped Product space. Then qualar curvature by using sectional curvatures of $M_{1}$ and $M_{2}$ is given by
		\begin{align}\label{8.4.11}
			\nonumber	\mathcal{Q}(p) = & 2\sum_{i = 1}^{q_{1}}\sum_{s = q_{1} + 1}^{k_{1}} \mathcal{K}^{M_{1}}(dx_{i}, dx_{s}) + 2(f_{1}^2 \nu_{1}^2)^h 	\sum_{i = 1}^{q_{2}}\sum_{s = q_{2} + 1}^{k_{2}} \mathcal{K}^{M_{2}}(dy_{i}, dy_{s})^v  \\
		\nonumber	&+ \left( \frac{3 f_{1}^4 ||d\nu_{1}||_{1}^2 }{2}\right) ^h \sum_{i = 1}^{q_{2}}\sum_{s = q_{2} + 1}^{k_{2}}|\tilde{g}_{2}\left(J_{2}dy_{i}, dy_{s} \right)^v|^2 - \frac{2 q_{2}}{f_{1}^h} \left( \Delta ^ {\mathcal{D}_{1}}(f_{1})\right)^h  \\
	\nonumber		&+  \frac{2 k_{2}}{f_{1}^h}\sum_{i = 1}^{q_{1}} H_{1}^{f_{1}}(dx_{i}, dx_{i})
	- \frac{4 q_{2}}{(f_{1}^h)^2} |\left( ||J_{1}df_{1}||_{1}^2\right) |^2 + \frac{4 k_{2}}{(f_{1}^h)^2} \sum_{i = 1}^{q_{1}}|\tilde{g}_{1}\left( J_{1}df_{1}, dx_{i}\right) |^2\\
	\nonumber		&- \sum_{i = 1}^{q_{1}}\sum_{s = q_{2} + 1}^{k_{2}} \frac{(f_{1}^h)^4}{2} |\tilde{g}_{1}\left( d\nu_{1}, dx_{i}\right)^h |^2 \left( || J_{2}dy_{s}, J_{2}dy_{s}||_{2}^2 \right)^v - 2 \sum_{i = 1}^{q_{2}}\sum_{s = q_{2} + 1}^{k_{2}}\frac{\left(\||J_{1}df_{1}||_{1}^2 \right)^h}{(f_{1}^h)^2} \\
			 &- \sum_{s = 1}^{q_{2}}\sum_{i = q_{1} + 1}^{k_{1}} \frac{(f_{1}^h)^4}{2} |\tilde{g}_{1}\left( d\nu_{1}, dx_{i}\right)^h |^2 \left( || J_{2}dy_{s}, J_{2}dy_{s}||_{2}^2 \right)^v. 
	\end{align}  
\end{theorem}  
\begin{proof}
Let  $(M= M_{1}\times_{f_{1}} M_{2}, \Pi^{\nu_1},\tilde{g}^{f_{1}})$, be a contravariant  warped Poisson space with local $\tilde{g}^{f_{1}}$-  orthonormal basis  $\lbrace dx^h_{1}, . . . dx^h_{k_{1}}, f_{1}^h dy_{1}^v, . . .  f_{k_1}^h dy_{k_{2}}^v \rbrace$, such that  $\lbrace dx_{1}, . . . dx_{k_{1}}\rbrace$ are $\tilde{g}_{1}$-  orthonormal basis in $M_{1}$ of index $q_{1}$ and  $\lbrace dy_{1}, . . .   dy_{k_{2}} \rbrace$ are  $\tilde{g}_{2}$-orthonormal basis in $M_{2}$ of index $q_{2}$. Then from Lemma \ref{lem3.1}, we have  
\begin{equation}\label{8.4.12}
\begin{cases}
{K}^M\left(dx_{i}^h,dx_{s}^h \right)  =  \mathcal{K}^{M_{1}}\left( dx_{i},dx_{s}\right) ^h, \; \forall\; i \in \lbrace 1, . . . q_{1} \rbrace , s \in \lbrace q_{1} + 1, . . . k_{1} \rbrace, \\

 \mathcal{K}^M\left(dx_{i}^h,f_{1}^h dy_{s}^v \right) =  \left(\frac{{\tilde{g}_{1}}\left( {\mathcal{D}^{M_1}_{dx_{i}}J_{1} df_{1}}, dx_{i} \right) }{f_{1} } \right)^h  + \frac{2}{(f_{1}^h)^2}  |{\tilde{g}_{1}}\left(J_{1}df_{1}, dx_{i} \right)^h |^2\\
 
 \qquad\;- \frac{1 }{4} \tilde{g}_{2}\left( J_{2}dy_{s}, J_{2}dy_{s}\right)^v |\left( f_{1}^2 \tilde{g}_{1}\left( d\nu_{1}, dx_{i}\right) \right)^h|^2 , \qquad \forall\; i \in \lbrace 1, . . . q_{1} \rbrace , s \in \lbrace q_{2} + 1, . . . k_{2} \rbrace,\\
 
 \mathcal{K}^M\left(dx_{i}^h,f_{1}^h dy_{s}^v \right) =  - \left(\frac{\tilde{g}_{1}\left( {\mathcal{D}^{M_1}_{dx_{i}}J_{1} df_{1}}, dx_{i} \right) }{f_{1} } \right)^h  - \frac{2}{(f_{1}^h)^2}  |\tilde{g}_{1}\left(J_{1}df_{1}, dx_{i} \right)^h |^2\\
 
 \qquad\;- \frac{1 }{4} \tilde{g}_{2}\left( J_{2}dy_{s}, J_{2}dy_{s}\right)^v |\left( f_{1}^2 \tilde{g}_{1}\left( d\nu_{1}, dx_{i}\right) \right)^h|^2 , \qquad \forall\; i \in \lbrace q_{1} + 1, . . . k_{1} \rbrace , s \in \lbrace  1, . . . q_{2} \rbrace,\\
 
 \mathcal{K}^M\left(dy_{i}^v,dy_{s}^v \right) = (\nu_{1}^h)^2 (f_{1}^h)^2  \mathcal{K}^{M_{2}}\left(dy_{i},dy_{s}\right)^v - \frac{\left(\||J_{1}df_{1}||_{1}^2 \right)^h}{(f_{1}^h)^2} \\
 
\qquad\qquad + \frac{3}{4}\left( f_{1}^4 ||d\nu_{1}||_{1}^2\right)^h |\tilde{g}_{2}\left( J_{2} dy_{i} , dy_{s}\right)^v|^2 \qquad \forall\; i \in \lbrace 1, . . . q_{2} \rbrace , s \in \lbrace  q_{2} + 1, . . . k_{2} \rbrace.
\end{cases}
\end{equation}
Thus using equation (\ref{8.4.12}), in equation (\ref{8.4.6}) and 
 $$\Delta ^ {\mathcal{D}_{1}}(f_{1}) = \sum_{i = 1}^{k_{1}} H_{1}^{f_{1}}(dx_{i}, dx_{i}) = \sum_{i = 1}^{k_{1}} g_{1}(\mathcal   {D}_{dx_{i}}^{M_1} J_{1}df_{1}, dx_{i}),$$
we obtain the required solution.
\end{proof}  
\begin{corollary}
Let $(M= M_{1}\times_{f_{1}} M_{2}, \Pi^{\nu_1},\tilde{g}^{f_{1}})$, be a contravariant   warped product space such that  $\nu_1 = constant$ and $M_{2}$ be a Riemannian manifold. Then qualar curvature on $M$ is independent of $M_{2}$. 
\end{corollary}
\begin{proof}
Taking $\nu_1 = constant$ and $M_{2}$ be a Riemannian manifold in Theorem \ref{x123}, the qualar curvature reduced to the relation
\begin{align}\label{8.4.13}
			\nonumber	\mathcal{Q}(p) =  2\sum_{i = 1}^{q_{1}}\sum_{s = q_{1} + 1}^{k_{1}} \mathcal{K}^{M_{1}}(dx_{i}, dx_{s}) 	
			\nonumber	&+  \frac{2 k_{2}}{f_{1}^h}\sum_{i = 1}^{q_{1}} H_{1}^{f_{1}}(dx_{i}, dx_{i})\\
	 &+ \frac{4 k_{2}}{(f_{1}^h)^2} \sum_{i = 1}^{q_{1}}|g_{1}\left( J_{1}df_{1}, dx_{i}\right) |^2.
	\end{align} 
Hence the proof. 
\end{proof}
\section{Examples} In this section, we find the sectional curvature for $H_{1}^2$, $E_{2}^2$ and $S_{0}^2$ by using Poisson tensor.  Also, the qualar curvatures for $(M= H_{1}^2\times_{f_{1}} E_{2}^2, \Pi^{\nu_1},\tilde{g}^{f_{1}})$ and $(M= H_{1}^2\times_{f_{1}} S_{0}^2, \Pi^{\nu_1},\tilde{g}^{f_{1}})$, are calculated.
\begin{example}
Let $(M= H_{1}^2\times_{f_{1}} E_{2}^2, \Pi^{\nu_1},\tilde{g}^{f_{1}})$, be a contravariant product  warped  space. Then for warped metric
$\bar{g}^{f_{1}} = -dx_{1}^2 + dx_{2}^2 + (f_{1}^h)^2 \left(-  d \bar{x}_{1}^2 -  d\bar{x}_{2}^2 \right)$,
  the local components of $\bar{g}^{f_{1}}$, are
  	\begin{eqnarray}
		\left\{
		\begin{array}{ll}
			\bar{g}_{11}=\bar{g}(\partial_1^h,\partial_1^h)=\bar{g}_{M_{1}}(\partial_1,\partial_1)^h=-1,\\
			
			\bar{g}_{22}=\bar{g}(\partial_2^h,\partial_2^h)=\bar{g}_{M_{1}}(\partial_2,\partial_2)^h=1,\\
			
			\bar{g}_{12}=\bar{g}(\partial_1^h,\partial_2^h)=0,\\
			
			\bar{g}_{1\bar{1}}=\bar{g}(\partial_1^h,\partial_{\bar{1}}^v)=0,\:\bar{g}_{2{\bar{1}}}=\bar{g}(\partial_2^h,\partial_{\bar{1}}^v)=0,\\
			
			\bar{g}_{1{\bar{2}}}=\bar{g}(\partial_1^h,\partial_{\bar{2}}^v)=0,\:\bar{g}_{2{\bar{2}}}=\bar{g}(\partial_2^h,\partial_{\bar{2}}^v)=0,\\
			
			\bar{g}_{{\bar{1}} {\bar{1}}}=\bar{g}(\partial_{\bar{1}}^v,\partial_{\bar{1}}^v)=({f_{1}}^h)^2\bar{g}_{M_{2}}(\partial_{\bar{1}},\partial_{\bar{1}})^v = - ({f_{1}}^h)^2,\\
			
			\bar{g}_{{\bar{2}} {\bar{2}}}=\bar{g}(\partial_{\bar{2}}^v,\partial_{\bar{2}}^v)=({f_{1}}^h)^2\bar{g}_{M_{2}}(\partial_{\bar{2}},\partial_{\bar{2}})^v = -({f_{1}}^h)^2 ,\\
			
			\bar{g}_{{\bar{1}} {\bar{2}} }=\bar{g}(\partial_{\bar{1}}^v,\partial_{\bar{2}}^v)=({f_{1}}^h)^2\bar{g}_{M_{2}}(\partial_{\bar{1}},\partial_{\bar{2}})^v = 0.
		\end{array}
		\right.
	\end{eqnarray} 
	Now, if $\tilde{g}^{f_{1}}$ is a cometric of  $\bar{g}^{f_{1}}$, then its local components are
  	\begin{eqnarray}
		\left\{
		\begin{array}{ll}
			\tilde{g}^{11}=\tilde{g}(dx_1^h,dx_1^h)=\tilde{g}_{M_{1}}(dx_1,dx_1)^h=-1,\\
			\tilde{g}^{22}=\tilde{g}(dx_2^h,dx_2^h)=\tilde{g}_{M_{1}}(dx_2,dx_2)^h=1,\\
			\tilde{g}^{12}=\tilde{g}(dx_1^h,dx_2^h)=0,\\
			\tilde{g}^{1{\bar{1}}}=\tilde{g}(dx_1^h,d{\bar{x}_1}^v)=0,\:\tilde{g}^{2{\bar{1}}}=\tilde{g}(dx_2^h,d{\bar{x}_1}^v)=0,\\
		\tilde{g}^{1{\bar{2}}}=\tilde{g}(dx_1^h,d{\bar{x}_2}^v)=0,\:\tilde{g}^{2 {\bar{1}}}=\tilde{g}(dx_2^h,d{\bar{x}_1}^v)=0,\\
			\tilde{g}^{{\bar{1}}{\bar{1}}}=\tilde{g}(d{\bar{x}_1}^v,d{\bar{x}_1}^v)=\frac{1}{({f_{1}}^h)^2}\tilde{g}_{M_{2}}(d{\bar{x}_1},d{\bar{x}_1})^v = - \frac{1}{({f_{1}}^h)^2},\\
			\tilde{g}^{\bar{2} \bar{2}}=\tilde{g}(d{\bar{x}_2}^v,d{\bar{x}_2}^v)=\frac{1}{(f^h)^2}\tilde{g}_{M_{2}}(d{\bar{x}_2},d{\bar{x}_2})^v = - \frac{1}{({f_{1}}^h)^2},\\
			\tilde{g}^{\bar{1}\bar{2}}=\tilde{g}(d{\bar{x}_1}^v,d{\bar{x}_2}^v)=\frac{1}{({f_{1}}^h)^2}\tilde{g}_{M_{2}}(d{\bar{x}_1},d{\bar{x}_2})^v = 0 = \tilde{g}^{\bar{2}\bar{1}}=\tilde{g}(d{\bar{x}_2}^v,d{\bar{x}_2}^v)
		\end{array}
		\right.
	\end{eqnarray}
	Let $\Pi_{H_{1}^2}=\Pi_{H_{1}^2}^{12}\frac{\partial}{\partial{x_1}}\wedge\frac{\partial}{\partial{x_2}}$, where $\Pi_{H_{1}^2}^{12} = \Pi_{H_{1}^2}(dx_{1}, dx_{2}) = \tilde{g}_{H_{1}^2} (J_{1}dx_{1}, dx_{2})$ and $\Pi_{E_{2}^2}=\Pi_{E_{2}^2}^{\bar{1}\bar{2}}\frac{\partial}{\partial{\bar{x}_{1}}}\wedge\frac{\partial}{\partial{\bar{x}_{2}}}$ are Poisson tensors on $M_{1}^2=H_{1}^2$ and $M_{2}^2=E_{2}^2$ respectively, such that  $\Pi=\Pi_{H_{1}^2}+\Pi_{S_{0}^2}$, $\nu = 1$ be a Poisson tensor on $(M = {H_{1}^2} \times _{f_{1}} {E_{2}^2},\tilde{g},\Pi)$. Then by using the relations \cite{8bpk}
	 
	\begin{align}\label{8.4.16}
\Gamma_k^{ij}&=\frac{1}{2}\sum_{l}\sum_{m}\tilde{g}_{mk}\Big(\Pi^{il}\frac{\partial \tilde{g}^{jm}}{\partial x_l}+\Pi^{jl}\frac{\partial \tilde{g}^{im}}{\partial x_l}-\Pi^{ml}\frac{\partial \tilde{g}^{ij}}{\partial x_l}-\tilde{g}^{li}\frac{\partial\Pi^{jm}}{\partial x_l}-\tilde{g}^{lj}\frac{\partial\Pi^{im}}{\partial x_l}\Big)\nonumber\\
&+\frac{1}{2}\frac{\partial\Pi^{ij}}{\partial x_k},
\end{align}
and,
\begin{equation}\label{8.4.17}
\mathcal{D}_{dx_i}dx_j=\Gamma_k^{ij}dx_k,
\end{equation}

we have obtained the Christoffel symbol and Levi-Civita connections associated to pairs $\left( \Pi_{H_{1}^2}, \tilde{g}_{H_{1}^2}\right) $  and $\left( \Pi_{E_{2}^2}, \tilde{g}_{E_{2}^2}\right) $. The Christoffel symbols on $\left( \Pi_{H_{1}^2}, \tilde{g}_{H_{1}^2}\right) $  and $\left( \Pi_{E_{2}^2}, \tilde{g}_{E_{2}^2}\right) $, are
\begin{align}\label{8.4.18}
&\Gamma_1^{11}=0,\quad\Gamma_1^{12}=\frac{\partial\Pi_{H_{1}^2}^{12}}{\partial x_1},\quad\Gamma_1^{21}=0,\quad\Gamma_1^{22}=-\frac{\partial\Pi_{H_{1}^2}^{12}}{\partial x_2},\nonumber\\
&\Gamma_2^{11}=\frac{\partial\Pi_{H_{1}^2}^{12}}{\partial x_1},\quad\Gamma_2^{12}=0,\quad\Gamma_2^{21}=-\frac{\partial\Pi_{H_{1}^2}^{12}}{\partial x_2},\quad\Gamma_2^{22}=0,
\end{align}
and
\begin{align}\label{8.4.19}
&\Gamma_{\bar{1}}^{{\bar{1}}{\bar{1}}}=0,\quad\Gamma_{\bar{1}}^{{\bar{1}}{\bar{2}}}=\frac{\partial\Pi_{E_{2}^2}^{{\bar{1}}{\bar{2}}}}{\partial \bar{x}_{1}},\quad\Gamma_{\bar{1}}^{{\bar{2}}{\bar{1}}}=0,\quad\Gamma_{\bar{1}}^{{\bar{2}}{\bar{2}}}= \frac{\partial\Pi_{E_{2}^2}^{{\bar{1}}{\bar{2}}}}{\partial  {\bar{x}_{2}}},\nonumber\\
&\Gamma_{\bar{2}}^{{\bar{1}}{\bar{1}}}= - \frac{\partial\Pi_{E_{2}^2}^{{\bar{1}}{\bar{2}}}}{\partial {\bar{x}_{1}}},\quad\Gamma_{\bar{2}}^{{\bar{1}}{\bar{2}}}=0,\quad\Gamma_{\bar{2}}^{{\bar{2}}{\bar{1}}}=-\frac{\partial\Pi_{E_{2}^2}^{{\bar{1}}{\bar{2}}}}{\partial {\bar{x}_{2}}},\quad\Gamma_{\bar{2}}^{{\bar{2}}{\bar{2}}}= 0.
\end{align}
 
The Levi-Civita connections for$\left( \Pi_{H_{1}^2}, \tilde{g}_{H_{1}^2}\right) $  and $\left( \Pi_{E_{2}^2}, \tilde{g}_{E_{2}^2}\right) $, from equations (\ref{8.4.17}), (\ref{8.4.18}) and (\ref{8.4.19}), are 
\begin{align}\label{8.4.20}
&\mathcal{D}^{H_{1}^2}_{dx_1}dx_1=\frac{\partial\Pi_{{H_{1}^2}}^{12}}{\partial x_1}dx_2,\quad\mathcal{D}^{H_{1}^2}_{dx_1}dx_2=\frac{\partial\Pi_{H_{1}^2}^{12}}{\partial x_1}dx_1,\nonumber\\
&\mathcal{D}^{H_{1}^2}_{dx_2}dx_1=-\frac{\partial\Pi_{H_{1}^2}^{12}}{\partial x_2}dx_2,\quad\mathcal{D}^{H_{1}^2}_{dx_2}dx_2=-\frac{\partial\Pi_{H_{1}^2}^{12}}{\partial x_2}dx_1,
\end{align}
and
\begin{align}\label{8.4.21}
&\mathcal{D}^{E_{2}^2}_{d{\bar{x}_1}}d{\bar{x}_1}= - \frac{\partial\Pi_{E_{2}^2}^{\bar{1} \bar{2}}}{\partial {\bar{x}_1}}d{\bar{x}_2},\quad\mathcal{D}^{E_{2}^2}_{d{\bar{x}_1}}d{\bar{x}_2}=\frac{\partial\Pi_{E_{2}^2}^{\bar{1} \bar{2}}}{\partial {\bar{x}_1}}d{\bar{x}_1},\nonumber\\
&\mathcal{D}^{E_{2}^2}_{d{\bar{x}_2}}d{\bar{x}_1}=-\frac{\partial\Pi_{E_{2}^2}^{\bar{1} \bar{2}}}{\partial {\bar{x}_2}}dt{\bar{x}_2},\quad\mathcal{D}^{E_{2}^2}_{d{\bar{x}_2}}d{\bar{x}_2}=\frac{\partial\Pi_{E_{2}^2}^{\bar{1}\bar{2}}}{\partial {\bar{x}_2}}d{\bar{x}_1}`.
\end{align}

Therefore from equation (\ref{8.2.2}), the sectional curvatures of $H_{1}^2$ for a plane spanned $\lbrace dx_{1}, dx_{2} \rbrace$, and  for a plane spanned $\lbrace d\bar{x}_{1}, d\bar{x}_{2} \rbrace$ of ${E_{2}^2}$ are
\begin{equation}\label{8.4.22}
\mathcal{K}^{H_{1}^2}(dx_{1}, dx_{2}) = \Pi^{12}_{H_{1}^2}\left[ \frac{\partial^2 \Pi^{12}_{H_{1}^2}}{\partial{ x_{1}}^2} - \frac{\partial^2 \Pi^{12}_{H_{1}^2}}{\partial{ x_{2}}^2}\right] - \left( \frac{\partial \Pi^{12}_{H_{1}^2}}{\partial{ x_{1}}}\right)^2 + \left( \frac{\partial \Pi^{12}_{H_{1}^2}}{\partial{ x_{2}}}\right)^2  
\end{equation}
and
\begin{equation}\label{8.4.23}
\mathcal{K}^{{E_{2}^2}}(d\bar{x}_{1}, d\bar{x}_{2}) = \left( \frac{\partial \Pi^{\bar{1}\bar{2}}_{E_{2}^2}}{\partial{ \bar{x}_{1}}}\right)^2 + \left( \frac{\partial \Pi^{\bar{1}\bar{2}}_{E_{2}^2}}{\partial{ \bar{x}_{2}}}\right)^2    - \Pi^{\bar{1}\bar{2}}_{E_{2}^2}\left[ \frac{\partial^2 \Pi^{\bar{1}\bar{2}}_{E_{2}^2}}{\partial{ \bar{x}_{1}}^2} + \frac{\partial^2 \Pi^{\bar{1}\bar{2}}_{E_{2}^2}}{\partial{ \bar{x}_{2}}^2}\right], 
\end{equation}
respectively. Also corresponding to $\tilde{g}^{H_{1}^2}$- orthonormal frame $\lbrace dx_{1}, dx_{2} \rbrace$, we have

 \begin{equation}\label{8.4.24}
 \begin{cases}
  J_{1}dx_{1} = \Pi^{12}_{H_{1}^2}dx_{2},  \; J_{1}dx_{2} = \Pi^{12}_{H_{1}^2}dx_{1},   \;J_{1}df_{1} = \Pi^{12}_{H_{1}^2}\left[\frac{\partial f_{1}}{\partial x_{1}} dx_{2} + \frac{\partial f_{1}}{\partial x_{2}} dx_{1} \right] ,\\\\
 \tilde{g}_{1}\left( \mathcal{D}_{dx_{1}}^1 J_{1} df_{1}, dx_{1}\right) = - \Pi^{12}_{H_{1}^2}\left[\frac{\partial \Pi^{12}_{H_{1}^2}}{\partial{ x_{1}}} \frac{\partial f_{1}}{\partial{ x_{1}}} + \frac{\partial \Pi^{12}_{H_{1}^2}}{\partial{ x_{2}}} \frac{\partial f_{1}}{\partial{ x_{2}}} + \Pi^{12}_{H_{1}^2} \frac{\partial^2 f_{1}}{\partial{ x_{2}}^2} .\right], \\\\
  \tilde{g}_{1}\left( \mathcal{D}_{dx_{2}}^1 J_{1} df_{1}, dx_{2}\right) = - \Pi^{12}_{H_{1}^2}\left[ \frac{\partial \Pi^{12}_{H_{1}^2}}{\partial{ x_{1}}} \frac{\partial f_{1}}{\partial{ x_{1}}} + \frac{\partial \Pi^{12}_{H_{1}^2}}{\partial{ x_{2}}} \frac{\partial f_{1}}{\partial{ x_{2}}} + \Pi^{12}_{H_{1}^2} \frac{\partial^2 f_{1}}{\partial{ x_{1}}^2} \right].
 \end{cases}
 \end{equation}
 Using the above equations in (\ref{8.4.11}), we get the qualar curvature at some point $p \in M$,
 \begin{align}\label{8.4.25}
\nonumber \mathcal{Q}(p) &= 2 \left\lbrace \Pi^{12}_{H_{1}^2}\left[ \frac{\partial^2 \Pi^{12}_{H_{1}^2}}{\partial{ x_{1}}^2} - \frac{\partial^2 \Pi^{12}_{H_{1}^2}}{\partial{ x_{2}}^2}\right] - \left( \frac{\partial \Pi^{12}_{H_{1}^2}}{\partial{ x_{1}}}\right)^2 + \left( \frac{\partial \Pi^{12}_{H_{1}^2}}{\partial{ x_{2}}}\right)^2 \right\rbrace^h \\
\nonumber &+ 2\left(f_{1}^2 \right)^h  \left\lbrace  \left( \frac{\partial \Pi^{\bar{1}\bar{2}}_{E_{2}^2}}{\partial{ \bar{x}_{1}}}\right)^2 + \left( \frac{\partial \Pi^{\bar{1}\bar{2}}_{E_{2}^2}}{\partial{ \bar{x}_{2}}}\right)^2    - \Pi^{\bar{1}\bar{2}}_{E_{2}^2}\left[ \frac{\partial^2 \Pi^{\bar{1}\bar{2}}_{E_{2}^2}}{\partial{ \bar{x}_{1}}^2} + \frac{\partial^2 \Pi^{\bar{1}\bar{2}}_{E_{2}^2}}{\partial{ \bar{x}_{2}}^2}\right]  \right\rbrace^v \\
\nonumber &+ \left( \frac{2 \Pi^{12}_{H_{1}^2}}{f_{1}^2} \right)^h  \left\lbrace 5 \left(  \frac{\partial f_{1}}{\partial{ x_{1}}}\right)^2  - \left(  \frac{\partial f_{1}}{\partial{ x_{2}}}\right)^2  + f_{1}^h\left(\frac{\partial^{2} f_{1}}{\partial x_{1}^2} - \frac{\partial^{2} f_{1}}{\partial x_{2}^2} \right)  \right\rbrace ^h\\
& - \left(\frac{4}{f_{1}^2} \right)^h  \left\lbrace (\Pi_{H_{1}^2}^{12})^4 \left( \frac{\partial^{2} f_{1}}{\partial x_{1}^2} - \frac{\partial^{2} f_{1}}{\partial x_{2}^2} \right)^2\right\rbrace^h.
 \end{align}
\end{example} 

\begin{remark}
 Let $(M= H_{1}^2\times_{f_{1}} E_{2}^2, \Pi^{\nu_1},\tilde{g}^{f_{1}})$, be a contravariant  warped product space such that $\Pi_{H_{1}^2}=\Pi_{H_{1}^2}^{12}\frac{\partial}{\partial{x_1}}\wedge\frac{\partial}{\partial{x_2}}$, $\Pi_{H_{1}^2}^{12} = c x_{1}$ and $\Pi_{E_{2}^2}=\Pi_{E_{2}^2}^{\bar{1}\bar{2}}\frac{\partial}{\partial{\bar{x}_{1}}}\wedge\frac{\partial}{\partial{\bar{x}_{2}}}$, $\Pi_{E_{2}^2}^{\bar{1}\bar{2}} = c \bar{x}_{1}$ are Poisson tensors on $H_{1}^2$ and $E_{2}^2$ respectively. Then equations (\ref{8.4.21}), (\ref{8.4.22}) and (\ref{8.4.23}) reduced to
$\mathcal{K}^1(d\bar{x}_{1}, d\bar{x}_{2}) = -c^2$, $\mathcal{K}^2(d\bar{x}_{1}, d\bar{x}_{2}) = c^2$, where $c$ is some constant, and qualar curvature is
\begin{align*}
\mathcal{Q}(p) = - 2 c^2 + 2 c^2 f_{1}^h +  2 c \left\lbrace 5 \left(  \frac{\partial f_{1}}{\partial{ x_{1}}}\right)^2  - \left(  \frac{\partial f_{1}}{\partial{ x_{2}}}\right)^2  + f_{1}^h\left(\frac{\partial^{2} f_{1}}{\partial x_{1}^2} - \frac{\partial^{2} f_{1}}{\partial x_{2}^2} \right)  \right\rbrace ^h\\
 - \left(\frac{4}{f_{1}^2} \right)^h  \left\lbrace (cx_{1})^4 \left( \frac{\partial^{2} f_{1}}{\partial x_{1}^2} - \frac{\partial^{2} f_{1}}{\partial x_{2}^2} \right)^2\right\rbrace^h. 
\end{align*}
If, we consider $f_{1} = constant = a$, then $\mathcal{Q}(p) = a_{1} \geq 0$, $a_{1} \in R$. 
\end{remark}

\begin{example}
Let $(M= H_{1}^2\times_{f_{1}} S_{0}^2, \Pi^{\nu_1},\tilde{g}^{f_{1}})$, be a contravariant  warped Product space. Then for  warped metric $\bar{g}^{f_{1}^h} = -dx_{1}^2 + dx_{2}^2 + (f_{1})^2 \left( d\theta^2 + sin^2(\theta) d\phi^2 \right)$,
  the local components of $\bar{g}^{f_{1}}$, are
  	\begin{eqnarray}
		\left\{
		\begin{array}{ll}
			\bar{g}_{11}=\bar{g}(\partial_1^h,\partial_1^h)=\bar{g}_{M_{1}}(\partial_1,\partial_1)^h=-1,\\
			
			\bar{g}_{22}=\bar{g}(\partial_2^h,\partial_2^h)=\bar{g}_{M_{1}}(\partial_2,\partial_2)^h=1,\\
			
			\bar{g}_{12}=\bar{g}(\partial_1^h,\partial_2^h)=0,\\
			
			\bar{g}_{1\theta}=\bar{g}(\partial_1^h,\partial_\theta^v)=0,\:\bar{g}_{2\theta}=\bar{g}(\partial_2^h,\partial_\theta^v)=0,\\
			
			\bar{g}_{1\phi}=\bar{g}(\partial_1^h,\partial_\phi^v)=0,\:\bar{g}_{2\phi}=\bar{g}(\partial_2^h,\partial_\phi^v)=0,\\
			
			\bar{g}_{\theta \theta}=\bar{g}(\partial_\theta^v,\partial_\theta^v)=({f_{1}}^h)^2\bar{g}_{M_{2}}(\partial_\theta,\partial_\theta)^v = ({f_{1}}^h)^2,\\
			
			\bar{g}_{\phi \phi}=\bar{g}(\partial_\phi^v,\partial_\phi^v)=({f_{1}}^h)^2\bar{g}_{M_{2}}(\partial_\phi,\partial_\phi)^v = ({f_{1}}^h)^2 (sin^2(\theta))^v,\\
			
			\bar{g}_{\theta \phi}=\bar{g}(\partial_\theta^v,\partial_\phi^v)=({f_{1}}^h)^2\bar{g}_{M_{2}}(\partial_\theta,\partial_\phi)^v.
		\end{array}
		\right.
	\end{eqnarray} 
	If $\bar{g}^{f_{1}}$ is a cometric of  $\tilde{g}^{f_{1}}$, then its local components are
  	\begin{eqnarray}
		\left\{
		\begin{array}{ll}
			\tilde{g}^{11}=\tilde{g}(dx_1^h,dx_1^h)=\tilde{g}_{M_{1}}(dx_1,dx_1)^h=-1,\\
			\tilde{g}^{22}=\tilde{g}(dx_2^h,dx_2^h)=\tilde{g}_{M_{1}}(dx_2,dx_2)^h=1,\\
			\tilde{g}^{12}=\tilde{g}(dx_1^h,dx_2^h)=0,\\
			\tilde{g}^{1\theta}=\tilde{g}(dx_1^h,d\theta^v)=0,\:\tilde{g}^{2\theta}=\tilde{g}(dx_2^h,d\theta^v)=0,\\
		\tilde{g}^{1\phi}=\tilde{g}(dx_1^h,d\phi^v)=0,\:\tilde{g}^{2\phi}=\tilde{g}(dx_2^h,d\phi^v)=0,\\
			\tilde{g}^{\theta\theta}=\tilde{g}(d\theta^v,d\theta^v)=\frac{1}{({f_{1}}^h)^2}\tilde{g}_{M_{2}}(d\theta,d\theta)^v = \frac{1}{({f_{1}}^h)^2},\\
			\tilde{g}^{\phi\phi}=\tilde{g}(d\phi^v,d\phi^v)=\frac{1}{(f^h)^2}\tilde{g}_{M_{2}}(d\phi,d\phi)^v = \frac{1}{({f_{1}}^h)^2(sin^2(\theta))},\\
			\tilde{g}^{\theta\phi}=\tilde{g}(d\theta^v,d\phi^v)=\frac{1}{({f_{1}}^h)^2}\tilde{g}_{M_{2}}(d\theta,d\phi)^v = 0 = \tilde{g}^{\theta\phi}=\tilde{g}(d\phi^v,d\theta^v).
		\end{array}
		\right.
	\end{eqnarray}
	Let $\Pi_{H_{1}^2}=\Pi_{H_{1}^2}^{12}\frac{\partial}{\partial{x_1}}\wedge\frac{\partial}{\partial{x_2}}$ and $\Pi_{S_{0}^2}=\Pi_{S_{0}^2}^{\theta \phi}\frac{\partial}{\partial{\theta}}\wedge\frac{\partial}{\partial{\phi}}$ Poisson tensors on $M_{1}^2=H_{1}^2$ and $M_{2}^2=S_{0}^2$ respectively, such that  $\Pi=\Pi_{H_{1}^2}+\Pi_{S_{0}^2}$, $\nu = 1$ be a Poisson tensor on $(M = {H_{1}^2} \times _{f_{1}} {S_{0}^2},\tilde{g},\Pi)$.  Then the Christoffel symbols on $\left( \Pi_{H_{1}^2}, \tilde{g}_{H_{1}^2}\right) $  and $\left( \Pi_{S_{0}^2}, \tilde{g}_{S_{0}^2}\right) $, are 
\begin{align}
&\Gamma_{\theta}^{{\theta}{\theta}}=0,\quad\Gamma_{\theta}^{{\theta}{\phi}}=\frac{\partial\Pi_{S_{0}^2}^{{\theta}{\phi}}}{\partial t_{\theta}},\quad\Gamma_{\theta}^{{\phi}{\theta}}=0,\quad\Gamma_{\theta}^{{\phi}{\phi}}=\frac{1}{\sin^2(\theta)}\frac{\partial\Pi_{S_{0}^2}^{{\theta}{\phi}}}{\partial  {\phi}},\nonumber\\
&\Gamma_{\phi}^{{\theta}{\theta}}= - \sin^2(\theta)\frac{\partial\Pi_{S_{0}^2}^{{\theta}{\phi}}}{\partial {\theta}},\quad\Gamma_{\phi}^{{\theta}{\phi}}=0,\quad\Gamma_{\phi}^{{\phi}{\theta}}=-\frac{\partial\Pi_{S_{0}^2}^{{\theta}{\phi}}}{\partial {\phi}},\quad\Gamma_{\phi}^{{\phi}{\phi}}=  \cot(\theta) \Pi_{S_{0}^2}^{{\theta}{\phi}}.
\end{align}
 The Levi-Civita connection on $\left( \Pi_{H_{1}^2}, \tilde{g}_{H_{1}^2}\right) $ is obtained in equation (\ref{8.4.19}), whereas for $\left( \Pi_{S_{0}^2}, \tilde{g}_{S_{0}^2}\right) $, we have
 \begin{align}\label{8.4.29}
&\mathcal{D}^{S_{0}^2}_{d{\theta}}d{\theta}= - sin^2(\theta)\frac{\partial\Pi_{S_{0}^2}^{\theta \phi}}{\partial {\theta}}d{\phi},
\quad\mathcal{D}^{S_{0}^2}_{d{\theta}}d{\phi}=\frac{\partial\Pi_{S_{0}^2}^{\theta \phi}}{\partial \theta}d{\theta},\nonumber\\
&\mathcal{D}^{S_{0}^2}_{d{\phi}}d{\theta}=-\frac{\partial\Pi_{S_{0}^2}^{\theta \phi}}{\partial {\phi}}dt{\phi},
\quad\mathcal{D}^{S_{0}^2}_{d{\phi}}d{\phi}=\frac{1}{sin^2(\theta)}\frac{\partial\Pi_{S_{0}^2}^{\theta\phi}}{\partial {\phi}}d{\theta} + cot(\theta) \Pi^{\theta \phi} d\phi.
\end{align}

Thus the sectional curvature of $S_{0}^2$ for a plane spanned by $\lbrace d\theta, d\phi \rbrace$,
 \begin{align}
 \nonumber \mathcal{K}^{S_{0}^2}\left( d\theta, d\phi\right) &= \Pi_{S_{0}^2}^{\theta\phi}\left[  sin^2(\theta)\frac{\partial^2\Pi_{S_{0}^2}^{\theta\phi}}{\partial \theta} + \frac{\partial^2\Pi_{S_{0}^2}^{\theta\phi}}{\partial \phi}  - \frac{sin(2\theta) }{2} \frac{\partial\Pi_{S_{0}^2}^{\theta\phi}}{\partial \theta} \right]\\
 &\qquad\qquad\qquad -\left[sin^2(\theta)\left(\frac{\partial\Pi_{S_{0}^2}^{\theta\phi}}{\partial \theta}\right)^2  + \left( \frac{\partial\Pi_{S_{0}^2}^{\theta\phi}}{\partial \phi}\right)^2  \right], 
 \end{align}
and, from equation (\ref{8.4.13}), we get  qualar curvature on $M$ 
\begin{align}\label{8.4.30}
\nonumber& \mathcal{Q}(p) = 2 \left\lbrace \Pi^{12}_{H_{1}^2}\left[ \frac{\partial^2 \Pi^{12}_{H_{1}^2}}{\partial{ x_{1}}^2} - \frac{\partial^2 \Pi^{12}_{H_{1}^2}}{\partial{ x_{2}}^2}\right] - \left( \frac{\partial \Pi^{12}_{H_{1}^2}}{\partial{ x_{1}}}\right)^2 + \left( \frac{\partial \Pi^{12}_{H_{1}^2}}{\partial{ x_{2}}}\right)^2 \right\rbrace^h \\
&- \left( \frac{4 \Pi^{12}_{H_{1}^2}}{f_{1}}\right)^h \left\lbrace \frac{\partial \Pi^{12}_{H_{1}^2}}{\partial{ x_{1}}} \frac{\partial f_{1}}{\partial{ x_{1}}} + \frac{\partial \Pi^{12}_{H_{1}^2}}{\partial{ x_{2}}} \frac{\partial f_{1}}{\partial{ x_{2}}} + \Pi^{12}_{H_{1}^2} \frac{\partial^2 f_{1}}{\partial{ x_{2}}^2} \right\rbrace^h  + \left( \frac{8 (\Pi^{12}_{H_{1}^2})^2 }{f_{1}^2} \right)^h \left\lbrace \left( \frac{\partial f_{1}}{\partial x_{2}}\right)^2 \right\rbrace^h  
 \end{align}

\end{example}

\end{document}